\documentclass[12pt,a4paper]{article}
\usepackage{amsfonts,amssymb,amsmath,amsbsy}
\usepackage{times}
\usepackage{a4wide}
\newtheorem{theo}{Theorem}[section]
\newtheorem{lemma}[theo]{Lemma}
\newtheorem{prop}[theo]{Proposition}
\newtheorem{coro}[theo]{Corollary}
\newcommand{\proof}{\noindent{\it Proof: }}
\newcommand{\proofbox}{\hfill \mbox{ $\Box$}\\}

\newcommand{\R}{{\mathbb R}}

\newcommand{\N}{{\mathbb N}}

\newcommand{\E}{\mathbb{E}}
\newcommand{\PP}{\mathbb{P}}
\newcommand{\NL}{\mathcal{N}(L)}
\newcommand{\mq}{{\mu}_q}
\newcommand{\Hde}{\mathcal{H}^{d-1}}

\date{}

\title{The mean width of random polytopes circumscribed around a convex body}

\author{K\'aroly J. B\"or\"oczky\footnote{Supported by OTKA grants 068398 
and 049301 and by the EU Marie Curie projects BudAlgGeo (MTKD-CT-2004-002988)
and DiscConvGeo (MTKD-CT-2005-014333), and by the J\'anos Bolyai Research 
Scholarship of the Hungarian
Academy of Sciences.}, Ferenc Fodor
\footnote{Supported by the J\'anos Bolyai Research Scholarship of the Hungarian
Academy of Sciences and by OTKA grant 068398.}, 
Daniel Hug\footnote{Supported by the 
European Network PHD, FP6 Marie Curie Actions, RTN, Contract MCRN -511953.\newline
{\bf Subjclass[2000]:} Primary 52A22, Secondary 60D05, 52A27
}}

\begin{document}


\maketitle

\begin{abstract}
Let $K$ be a $d$-dimensional convex body, and let $K^{(n)}$ be the intersection
of $n$ halfspaces containing $K$ whose bounding hyperplanes are 
independent and 
identically distributed. Under suitable distributional assumptions, 
we prove an asymptotic formula for the expectation of the
difference of the mean widths of $K^{(n)}$ and $K$, and 
another asymptotic formula 
for the expectation of the number of facets of $K^{(n)}$. These results
are achieved by establishing an asymptotic result on  
weighted volume approximation  of $K$ and by ``dualizing'' it using
polarity. 
\end{abstract}

\section{Introduction}

Let $K$ be a convex body (compact convex set with nonempty interior) in
$d$-dimensional Euclidean space $\R^d$. The convex hull $K_{(n)}$ of $n$ 
independent random points in $K$ chosen according to the uniform 
distribution is a common model of a  random polytope contained in $K$. 
The famous four-point problem of Sylvester \cite{S1864} is the 
starting point of an extensive investigation of random polytopes of 
this type. Beside specific probabilities as in Sylvester's problem, 
important objects of study are expectations, variances and distributions of various 
geometric functionals associated with $K_{(n)}$. Typical examples of such 
functionals are volume, other intrinsic
volumes, and the number of $i$-dimensional faces. 
In their ground-braking papers \cite{RS63, RS64}, R\'enyi and Sulanke  considered 
random polytopes in the Euclidean plane and proved asymptotic results for the expectations 
of basic functionals of random polytopes in a convex domain $K$ in the cases where $K$ is 
sufficiently smooth or a convex polygon.  
Since then most results have been in the form of asymptotic formulae as the number $n$ of 
random points tends to infinity. In the last three decades, much effort has been devoted to
exploring the properties of this particular model of a random polytope contained in a $d$-dimensional
convex body $K$. For instance, for a sufficiently smooth convex body $K$, 
asymptotic formulae were proved
for the expectation of the mean width
difference $W(K)-W(K_{(n)})$ by Schneider and  Wieacker \cite{ScW80}, 
and for the volume difference $V(K)-V(K_{(n)})$ by B\'ar\'any \cite{Bar92}. 
The assumption of smoothness was
relaxed in the case of the mean width by 
B\"or\"oczky, Fodor, Reitzner and V\'{\i}gh \cite{BFRV},
and removed by
Sch\"utt \cite{Sch94} in the case of the volume.
Recently, even
variance estimates, laws of large numbers, and central limit theorems 
have been proved in a sequence of contributions, for instance by B\'ar\'any, Reitzner, and Vu. 
For more details on the current state-of-the-art of this line of research, 
see the survey papers by Weil and Wieacker \cite{WW93}, 
 Gruber \cite{Gru97} and 
Schneider \cite{Sch04}, and the recent monograph of Schneider and Weil \cite{SW}.

In a third paper, R\'enyi and Sulanke \cite{RS68} considered a ``dual'' 
model of a random polytope contained in a given convex body $K$ (a random inscribed polytope), 
that is a random polytope containing a given convex body (a random circumscribed polytope). 
Subsequently, this approach has not received nearly as much
attention as the ``inscribed case''. 
There are various ways of producing circumscribed
random polytopes containing a given convex body. 
In this paper, we consider a model in which
the circumscribed polytope arises as an 
intersection of closed halfspaces whose bounding hyperplanes are randomly chosen 
hyperplanes. The rough description of the probability model is the following, 
it is described more precisely in Section \ref{Sec:2}, a more general setting is 
provided in Section \ref{secpolar}. In Euclidean space $\R^d$, we consider hyperplanes that intersect
the radius one parallel domain of a given convex body $K$ but miss the interior of $K$, 
and we use the restriction of the 
(suitably normalized) Haar measure on the set of hyperplanes 
in $\R^d$ to provide an associated  
probability measure. For  $n$ independent random hyperplanes $H_1,\ldots,H_n$ chosen
according to this distribution, the 
intersection of the closed halfspaces bounded by $H_1,\ldots,H_n$ and 
containing $K$ determines a circumscribed random polyhedral set 
containing  $K$ (which might be unbounded). 
The main goal of this article is to find asymptotic formulae for 
the expectation of the difference of the mean widths of a random circumscribed polytope 
and the given convex body $K$, 
and for the expectation of the number of facets of a circumscribed random polyhedral set. 
These (and more general) results will be achieved by establishing general results on weighted volume 
approximation of a given convex body by inscribed random polytopes. In all these results, no regularity 
or curvature assumptions on $K$ are requird.

As for earlier results, we mention the paper \cite{Zie70} 
by Ziezold  who investigated circumscribed polygons in the plane, 
and the doctoral dissertation \cite{Kal90} 
of Kaltenbach  who proved asymptotic formulae for 
the expectations of the volume difference and the
number of vertices of circumscribed random polytopes around a 
convex body under the assumption that
the boundary of $K$ is sufficiently smooth. 
Recently, B\"or\"oczky and Schneider \cite{BS08} established 
upper and lower bounds for the expectation of the mean width 
difference for general convex bodies. 
Furthermore, they also proved asymptotic formulae for the 
expected number of vertices and facets of $K^{(n)}$, 
and an asymptotic formula for the expectation of the mean width difference, 
under the assumption that the parent body $K$ is a simplicial 
polytope with $r$ facets. 

In \cite{BR04},  B\"or\"oczky and Reitzner 
discuss a different model of a random circumscribed polytope 
where $n$ independent random points are chosen from the boundary of $K$, 
and the intersection of the supporting halfspaces of $K$ at these points 
is the random polyhedral set under consideration. This 
framework is again dual to the one considered by Sch\"utt and Werner (see \cite{SW}) who study  
the expected volume of the convex hull of $n$ independent random points chosen 
from the boundary of a sufficiently regular convex body.

\section{The probability space and the main goal}\label{Sec:2}

Let us first describe the setting for stating our results on circumscribed 
random polyhedral sets.  
Throughout this article, $K$ will denote a compact convex set 
with interior points (a convex body) in $d$-dimen\-sion\-al 
Euclidean space ${\mathbb R}^d$ ($d\ge 2$). We write $\langle\cdot\,,\cdot\rangle$ 
for the scalar product and $\|\cdot\|$ for the norm in $\R^d$. 
For any notions on convexity, we refer to the
monographs by Schneider \cite{Sch93} or  by Gruber \cite{Gru07}. 
Let $V$ denote volume, and let 
${\cal H}^{j}$ denote the $j$-dimensional Hausdorff measure. 
The unit ball of ${\mathbb R}^d$ with center at the
origin is denoted by $B^d$, and $S^{d-1}$ is its boundary. We put $\alpha_d:=V(B^d)$  
and $\omega_d:={\cal H}^{d-1}(S^{d-1})=d\alpha_d$. 
The parallel body of $K$ of radius $1$ 
is $K_1:=K+B^d$. Let ${\mathcal H}$ denote the space (with its usual
topology) of hyperplanes in ${\mathbb R}^d$, and let ${\mathcal H}_K$ be 
 the subspace of hyperplanes meeting $K_1$ but not the interior
of $K$. For $H\in {\mathcal H}_K$, the closed halfspace bounded by
$H$ that contains $K$ is denoted by $H^-$. Let $\mu$ denote 
the motion invariant Borel measure on ${\mathcal H}$, normalized
so that $\mu(\{H\in {\mathcal H}: H\cap M\not=\emptyset\})$ is the
mean width $W(M)$ of $M$, for every convex body $M\subset {\mathbb
R}^d$. Let $2\mu_K$ be the  restriction of $\mu$ to ${\mathcal
H}_K$. Since $\mu({\mathcal H}_K)= W(K+B^d)-W(K)= W(B^d)=2$, the
measure $\mu_K$ is a probability measure. For $n\in{\mathbb N}$,
let $H_1,\dots,H_n$ be independent random hyperplanes in ${\mathbb
R}^d$, i.e.\ independent ${\mathcal H}$-valued random variables on some probability
space $(\Omega,{\bf A},{\mathbb P})$, each with distribution
$\mu_K$. The possibly unbounded intersection
$$
K^{(n)}:= \bigcap_{i=1}^n H_i^-
$$
of the halfspaces $H^-_i$, with $H_i\in \mathcal{H}_K$ for $i=1,\ldots,n$, is a random
polyhedral set. A major aim of the present work is to 
investigate ${\mathbb E}W(K^{(n)}\cap K_1)$, where ${\mathbb E}$ denotes mathematical
expectation. The intersection with $K_1$ is considered, since $K^{(n)}$ 
is unbounded with positive probability. Instead of ${\mathbb E}W(K^{(n)}\cap K_1)$, 
we could consider ${\mathbb E}_1
W(K^{(n)})$, the conditional expectation of $W(K^{(n)})$ under the
condition that $K^{(n)}\subset K_1$. Since ${\mathbb
E}W(K^{(n)}\cap K_1)={\mathbb E}_1 W(K^{(n)})+O(\gamma^n)$ with $\gamma
\in (0,1)$ (cf.\ \cite{BS08}), there is no difference in the
asymptotic behaviors of both quantities, as $n\to\infty$. We also
remark that, for the asymptotic results, the parallel body $K_1$
could be replaced by any other convex body containing $K$ in its
interior; this would only affect some normalization constants.

Let $\partial K$ denote the boundary of $K$. 
We call $\partial K$ twice differentiable in the generalized sense at
a boundary point $x\in\partial K$ if there exists a quadratic form 
$Q$ on $\R^{d-1}$, the second fundamental form of $K$ at $x$,
with the following property: If $K$ is positioned in such a way that $x=o$
and $\R^{d-1}$ is a support hyperplane of $K$ at $o$,  then in a neighborhood
of $o$,  $\partial K$ is the graph of a convex function $f$ defined on a
$(d-1)$-dimensional ball around $o$ in  $\R^{d-1}$ satisfying
\begin{equation}
\label{diff}
f(z)=\mbox{$\frac12$}\,Q(z)+o(\|z\|^2),
\end{equation}
as $z\to o$. Alternatively, we call $x$ a normal boundary point of $K$. 
If this is the case, we write $\kappa(x)=\det(Q)$ to denote
the generalized Gaussian curvature of $K$ at $x$.
Writing $\kappa(x)$, we always assume that
$\partial K$ is twice differentiable in the generalized sense at
$x\in\partial K$.
According to a classical result of Alexandrov
(see \cite{Sch93}, \cite{Gru07}), $\partial K$ is twice differentiable in the generalized 
sense almost everywhere with respect to the boundary measure of $K$ ($\Hde$ almost 
all boundary points are normal boundary points). 
Finally, we define the constant
\begin{equation}
\label{wieackercons}
c_d = \frac{(d^2+d+2)(d^2+1)}{2(d+3)\cdot (d+1)!} 
\Gamma\left(\frac{d^2+1}{d+1}\right) 
\left(\frac{d+1}{\alpha_{d-1}}\right)^{2/(d+1)}
\end{equation}
(cf.\  J.A. Wieacker) \cite{Wie78}, which will appear in the statements of our main results. 
 In the following, we simply write $dx$ instead of $\mathcal{H}^d(dx)$.
 
The main asymptotic result concerning the expected difference of the mean widths of $K^{(n)}$ and $K$ 
is the following theorem. Generalizations of Theorem \ref{mainmean}, and also 
of Theorem \ref{extfacets} below, which 
hold under more general distributional assumptions, are provided in Section \ref{secpolar}. There we also indicate 
the connection to the $p$-affine surface area of a convex body.

\begin{theo}
\label{mainmean} 
If $K$ is a convex body in $\R^d$,  then
$$
\lim_{n\to\infty}n^{\frac2{d+1}}\,\E(W(K^{(n)}\cap K_1)-W(K))= 
{2\,c_d}\,{\omega_d}^{-\frac{d-1}{d+1}}\, \int_{\partial K}
\kappa(x)^{\frac{d}{d+1}}\,\Hde(dx).
$$
\end{theo}

\bigskip

Let $f_i(P)$, $i\in \{0,\ldots,d-1\}$, denote the number of $i$-dimensional faces of a polyhedral set  $P$. 
In the statement of the following theorem, $K^{(n)}$ could be replaced by the intersection of $K^{(n)}$ 
with a fixed polytope containing $K$ in its interior without 
changing the right-hand side. Alternatively, instead of $\E(f_{d-1}(K^{(n)}))$ we could consider the 
conditional expectation of $ f_{d-1}(K^{(n)})$ under the assumption that $K^{(n)}$ is contained in $K_1$.

\bigskip

\begin{theo}\label{extfacets}
If $K$ is a convex body in $\R^d$,  then
$$
\lim_{n\to\infty}n^{-\frac{d-1}{d+1}}\,\E(f_{d-1}(K^{(n)}))= 
{c}_d\, \omega_d^{-\frac{d-1}{d+1}} \int_{\partial K}
\kappa(x)^{\frac{d}{d+1}}\,\Hde(dx).
$$
\end{theo}

Both theorems will be deduced from a ``dual'' result on  weighted volume approximation of 
convex bodies by inscribed random polytopes which is stated in the subsequent section. 

\section{Weighted volume approximation by inscribed polytopes}\label{Sec:3}

For a given convex body, we introduce a class of inscribed random polytopes.  
Let $C$ be a convex body in $\R^d$, let $\varrho$ be a bounded, nonnegative, 
measurable function on $C$, and let 
$\mathcal{H}^d\llcorner C$ denote the restriction of $\mathcal{H}^d$ to $C$. 
Assuming that $\int_C\varrho(x)\,\mathcal{H}^d(dx)>0$,
we choose random points 
from $C$ according to the probability 
measure 
$$
\PP_{\varrho,C}:=\left(\int_C\varrho(x)\,dx\right)^{-1}\varrho\,\mathcal{H}^d\llcorner C.
$$   
Expectation with respect to $\PP_{\varrho,C}$  
is denoted by $\E_{\varrho,C}$. 
The convex hull of $n$ independent and identically distributed random points with 
distribution $\PP_{\varrho,C}$ is denoted by $C_{(n)}$ if $\varrho$ is clear from the context. 
This yields a general model of an inscribed random polytope.

Generalizing  a result by C. Sch\"utt \cite{Sch94}, we prove the following theorem.

\begin{theo}
\label{weighted} 
For a convex body $K$ in $\R^d$, a probability
density function $\varrho$ on $K$, and an integrable 
function $\lambda:K\to\R$ such that, on a
neighborhood of $\partial K$ with respect to $K$, 
$\lambda$ and $\varrho$ are continuous and $\varrho$ is positive, 
\begin{equation}\label{stern1}
\lim_{n\to\infty}n^{\frac2{d+1}}\,
\E_{\varrho,K}\int_{K\setminus K_{(n)}}\lambda(x)\,dx= c_d 
\int_{\partial K}\varrho(x)^{\frac{-2}{d+1}}\lambda(x)
\kappa(x)^{\frac1{d+1}}\,\Hde(dx)
\end{equation}
where $c_d$ is defined in (\ref{wieackercons}).
\end{theo}

The limit on the right-hand side of \eqref{stern1} depends only on the values of $\varrho$ and $\lambda$ 
on the boundary of $K$. In particular, we may prescribe any continuous, positive function 
$\varrho$ on $\partial K$. Then any continuous extension of $\varrho$ to a probability density on $K$ 
(there always exists such an extension) 
will satisfy Theorem \ref{weighted} with the prescribed values of $\varrho$ on the right-hand side.

Our proof of Theorem \ref{weighted} 
is inspired by the argument in C. Sch\"utt \cite{Sch94} who considered the 
special case $\varrho\equiv\lambda\equiv 1$. 
We note that for Lemma~2 in \cite{Sch94}, which is crucial for the proof in \cite{Sch94}, 
no explicit proof is provided, but  
reference is given to an analogous result in an unpublished
note by M. Schmuckenschl\"ager. Besides a missing factor
$\frac12$, Lemma~2 does not hold in the generality stated in \cite{Sch94}. For instance, 
it is not true for simplices. Most probably, this gap can be overcome, but still our approach
to prove Theorem~\ref{weighted}, 
where  Lemma~2 in \cite{Sch94}
is replaced by the elementary
Lemma \ref{integration}, 
might be of some interest. 

The present partially new approach to Theorem \ref{weighted} 
involves also some other interesting new features. In particular, we do not need the concept 
of a Macbeath region. An outline of the proof is given below. 
It should also be emphasized that 
the generality of Theorem \ref{weighted} is needed for our study of 
circumscribed random polyhedral sets 
via duality.

\bigskip

\noindent
A classical argument going back to Efron shows that
$$
\E_{\varrho,K}\left(f_0(K_{(n)})\right)=n\cdot\E_{\varrho,K}\int_{K\setminus K_{(n-1)}}
\varrho(x)\, dx,
$$
which yields the following consequence of Theorem \ref{weighted}.

\begin{coro}
\label{weightedcor} 
For a convex body $K$ in $\R^d$, and for a probability
density function $\varrho$ on $K$ which is  continuous and positive in a
neighborhood of $\partial K$ with respect to $K$, 
$$
\lim_{n\to\infty}n^{-\frac{d-1}{d+1}}\,
\E_{\varrho,K} (f_0(K_{(n)}))= {c}_d 
\int_{\partial K}\varrho(x)^{\frac{d-1}{d+1}}
\kappa(x)^{\frac1{d+1}}\,\Hde(dx)
$$
where $c_d$ is defined in (\ref{wieackercons}).
\end{coro}

The proof of Theorem \ref{weighted} is obtained through the following intermediate steps.  
Details are provided in Section \ref{Proofweighted}. 
Since the convex body $K$ is fixed, we write $\E_\varrho$ and $\PP_\varrho$ 
instead of  $\E_{\varrho,K}$
and $\PP_{\varrho,K}$, respectively. 
The basic observation to prove Theorem~\ref{weighted} is that
\begin{equation}\label{basic1}
\E_\varrho\int_{K\setminus K_{(n)}}\lambda(x)\,dx= \int_K
\PP_\varrho\left(x\not\in K_{(n)} \right) \lambda(x)\,dx,
\end{equation}
which is an immediate consequence of Fubini's theorem. Throughout the proof, 
we may assume that $o\in{\rm int}(K)$. 
The asymptotic behavior, as $n\to\infty$,  
of the right-hand side of \eqref{basic1} is determined by points $x\in K$ 
which are sufficiently close 
to the boundary of $K$. In order to give this statement a precise meaning, scaled copies of $K$ are introduced 
as follows. For $t\in(0,1)$, we define 
$K_t:=(1-t)K$ and $y_t:=(1-t)y$ for $y\in\partial K$. 
In Lemma  \ref{closetoK}, we show that 
$$
\lim_{n\to\infty}n^{\frac{2}{d+1}}\, \int_{K_{n^{\frac{-1}{d+1}}}}
\PP_\varrho\left(x\not\in K_{(n)} \right)\lambda(x)\,dx=0.
$$
This limit relation is based on a geometric estimate of $\PP_\varrho\left(x\not\in K_{(n)} \right)$, 
provided in Lemma  \ref{genup}, and on a disintegration result stated as Lemma \ref{integration}. 

For $y\in\partial K$, we write $u(y)$ for some 
exterior unit normal of $K$ at $y$. This exterior unit normal 
is uniquely determined for $\mathcal{H}^{d-1}$ almost all boundary points of $K$. Applying the 
disintegration result again and using Lebesgue's dominated convergence result, we finally get
$$
\lim_{n\to\infty}n^{\frac2{d+1}}\E_\varrho
\int_{K\setminus K_{(n)}}\lambda(x)\,dx=
\int_{\partial K}\lambda(y)J_\varrho(y)\,\mathcal{H}^{d-1}(dy),
$$
where
$$
J_\varrho(y)=\lim_{n\to\infty}
\int_0^{n^{\frac{-1}{d+1}}}n^{\frac2{d+1}}\langle y,u(y)
\rangle\PP_\varrho\left(y_t\not\in K_{(n)} \right)\,dt
$$
for $\Hde$ almost all $y\in\partial K$. For the subsequent analysis, it is sufficient to consider 
a small cap of $K$ at a normal boundary point $y\in \partial K$. The case $\kappa(y)=0$ is 
treated in Lemma \ref{zerocurv}. The main case is $\kappa(y)>0$. Here we reparametrize $y_t$ 
as $\tilde{y}_s$, in terms of the probability content of a small cap of $K$ whose 
bounding hyperplane passes through $y_t$. This implies that 
$$
J_\varrho(y)=
(d+1)^{-\frac{d-1}{d+1}}\alpha_{d-1}^{-\frac{2}{d+1}}\varrho(y)^{\frac{-2}{d+1}}
 \kappa(y)^{\frac1{d+1}}
\lim_{n\to\infty}\int_0^{n^{-1/2}}n^{\frac2{d+1}}
\PP_\varrho\left(\tilde{y}_s\not\in K_{(n)} \right)s^{-\frac{d-1}{d+1}}\,ds,
$$
cf.\ \eqref{limitform0}. 
It is then a crucial step in the proof to show that the remaining integral 
asymptotically is independent of the particular convex body $K$, and thus the 
limit of the integral is the same as for a Euclidean ball (see Lemma \ref{compareball}). 
To achieve this, the integral is first approximated,  
up to a prescribed error of order $\varepsilon>0$, by replacing 
$\PP_\varrho\left(\tilde{y}_s\not\in K_{(n)} \right)$ by the probability of an event that 
depends only on a small cap of $K$ at $y$ and on a small number of random points. This important step is 
accomplished in Lemma \ref{pointsneary}. For the proofs of Lemmas 
\ref{pointsneary} and \ref{compareball} it is essential that 
the boundary of $K$ near the normal boundary point $y$ can be suitably approximated 
by the osculating paraboloid of $K$ at $y$.

\section{Proof of Theorem \ref{weighted}}\label{Proofweighted}

To start with the actual proof, we fix some further notation. 
For $y\in\partial K$ and $t\in (0,1)$, we define the cap 
$C(y,t):=\{x\in K:\langle u(y),x\rangle\geq \langle u(y),y_t\rangle\}$ 
whose bounding hyperplane  passes through $y_t$ and has normal $u(y)$. 
For $u\in\R^d\setminus\{o\}$ and $t\in\R$, we define the hyperplane 
$H(u,t):=\{x\in\R^d:
\langle x,u\rangle =t\}$, and the closed halfspaces $H^+(u,t):=\{x\in\R^d:
\langle x,u\rangle \ge t\}$ and $H^-(u,t):=\{x\in\R^d:
\langle x,u\rangle \le t\}$ bounded by $H(u,t)$. We denote by $h(K,\cdot)=h_K$ the 
support function of $K$, that is $h(K,u):=\max\{\langle x,u\rangle:x\in K\}$ for $u\in\R^d$.

For $y\in\partial K$, the maximal number 
$r\geq 0$ such that $y-ru(y)+rB^d\subset K$ is denoted by $r(y)$. 
This number is called the interior reach of the boundary point $y$. 
It is well known that $r(y)>0$ for $\mathcal{H}^{d-1}$ almost all $y\in \partial K$. 
If $r(y)>0$,  there is a unique tangent plane of $K$ at $y$. 
In particular, $r(y)\leq r(K)$ where $r(K)$
is the inradius  of $K$. The convex hull of 
subsets $X_1,\ldots,X_r\subset\R^d$ and points $z_1,\ldots,z_s\in\R^d$ is denoted by 
$[X_1,\ldots,X_r,z_1,\ldots,z_s]$.

For real functions $f$ and $g$ defined on the same space $I$, we write
$f\ll g$ or $f=O(g)$ if there exists a positive constant $\gamma$, depending
only on $K$, $\varrho$ and $\lambda$, such that 
$|f|\leq \gamma\cdot g$ on $I$. In general, we write $\gamma_0,\gamma_1,\ldots$ to
denote positive constants depending only on $K$, $\varrho$ and
$\lambda$. The Landau symbol $o(\cdot)$ is defined as usual. We further put $\R^+:=[0,\infty)$. 

Finally, we observe that there exists a constant 
$\gamma_0\in(0,1)$ such that for $y\in\partial K$, we
have
\begin{equation}
\label{yu(y)} 
|\langle y,u(y)\rangle|\geq \gamma_0\|y\|, 
\mbox{ and hence }\|y|u(y)^\bot\|\leq\sqrt{1-\gamma_0^2}\cdot \|y\|,
\end{equation}
where $y|u^\bot$ denotes the orthogonal projection of $y$ onto the orthogonal 
complement of the vector $u\in\R^d\setminus\{o\}$. Subsequently, we always assume that $n\in\N$.

\begin{lemma}
\label{genup}
There exists a constant $\delta>0$, depending on $K$ and $\varrho$, such that
if $y\in \partial K$ and $t\in (0,\delta)$, then
$$
\PP_\varrho\left(y_t\not\in K_{(n)} \right)\ll 
\left(1-\gamma_1r(y)^{\frac{d-1}2}t^{\frac{d+1}2}\right)^n.
$$
\end{lemma}

\noindent
{\bf Remarks } \begin{enumerate}
\item In addition, we may assume that on $K\setminus{\rm int}(K_\delta)$, 
both functions $\varrho,\lambda$
are continuous, $\varrho$ is positive and 
$\gamma_1r(K)^{\frac{d-1}2}\delta^{\frac{d+1}2}<1$.
\item In the following, we will use the notion of a ``coordinate corner''. 
Given an orthonormal basis in a linear $i$-dimensional subspace $L$, 
the corresponding $(i-1)$-dimensional coordinate planes cut $L$ into 
$2^i$ convex cones, which we call coordinate corners (with respect to $L$ 
and the given basis).
\end{enumerate}

\bigskip

\begin{proof} If $r(y)=0$, then there is nothing to prove. So let $r(y)>0$, thence $u(y)$ is 
uniquely determined. Choose an orthonormal basis in $u(y)^\perp$, and 
let $\Theta'_1,\ldots,\Theta'_{2^{d-1}}$ be the corresponding coordinate corners 
in $u(y)^\bot$. For $i=1,\ldots,2^{d-1}$ and $t\in[0,1]$, we define
$$
\Theta_{i,t}:=C(y,t)\cap\left(y_t+\left[\Theta'_i,\R^+y\right]
\right).
$$
If $\delta>0$ is small enough to ensure that $\varrho>0$ is positive and continuous 
in a neighborhood (relative to $K$) of $\partial K$, then
$$
\int_{\Theta_{i,t}}\varrho(x)\,dx\ge\gamma_2\, V(\Theta_{i,t}).
$$

If $y_t\not\in K_{(n)}$ and $o\in K_{(n)}$, then there exists a 
hyperplane $H$ through $y_t$, bounding the 
halfspaces $H^-$ and $H^+$, for which 
$K_{(n)}\subset H^-$. Moreover, there is some $i\in\{1,\ldots,2^{d-1}\}$ 
such that $\Theta_{i,t}\subset H^+$. Therefore
\begin{equation}\label{eqa}
\PP_\varrho\left(y_t\not\in K_{(n)}, o\in K_{(n)} \right)\ll
\sum_{i=1}^{2^{d-1}}\left(1-\gamma_2 V(\Theta_{i,t})\right)^n.
\end{equation}
Finally, we prove
\begin{equation}\label{eqb}
V(\Theta_{i,t})\gg r(y)^{\frac{d-1}2}t^{\frac{d+1}2},
\end{equation}
for $i=1,\ldots,2^{d-1}$.  According
to (\ref{yu(y)}), there exist positive constants $\gamma_3,\gamma_4$ with 
$\gamma_3\le 1$ such that
if $t\leq \gamma_3r(y)$, then $(y_t+\Theta_i')\cap K$ contains a
$(d-1)$-ball of radius at least
$$\gamma_4\sqrt{r(y)^2-(r(y)-t)^2}\ge \gamma_4\sqrt{r(y)t},
$$
and we are done.
On the other hand, if $t\geq \gamma_3r(y)$, then
$$
V(\Theta_{i,t})\gg t^d\gg r(y)^{\frac{d-1}2}t^{\frac{d+1}2}.
$$
To deal with the case $o\not\in K_{(n)}$, we observe that there exists 
a positive constant $\gamma_5\in (0,1)$ such that the probability measure 
of each of the $2^d$ coordinate corners of $\R^d$ is at least $\gamma_5$. 
If $o\not\in K_{(n)}$, then $\{x_1,\ldots,x_n\}$ is disjoint from one of 
these coordinate corners, and hence 
\begin{equation}\label{except}
\PP_\varrho(o\not\in K_{(n)})\le 2^d(1-\gamma_5)^n.
\end{equation}
Now the assertion follows from \eqref{eqa}, \eqref{eqb} and \eqref{except}. $\Box$
\end{proof}

\bigskip

\noindent
Subsequently, the estimate of Lemma \ref{genup} will be used, for instance, to restrict the 
domain of integration on the right-hand side of \eqref{basic1} (cf.\ Lemma \ref{closetoK}) 
and to justify an application of Lebesgue's 
dominated convergence theorem (see \eqref{limitform}). For these applications, 
we also need that if $c>0$ is such that 
$\omega:=c\,\delta^{\frac{d+1}2}<1$, 
then 
\begin{equation}
 \label{Gamma}
\int_0^\delta\left(1-c\,t^{\frac{d+1}2}\right)^n\, dt
= \frac2{d+1}\, c^{\frac{-2}{d+1}}
\int_0^{\omega}s^{\frac2{d+1}-1}(1-s)^n\,ds
\ll  c^{\frac{-2}{d+1}}\cdot n^{\frac{-2}{d+1}},
\end{equation}
where we use that $(1-s)^{n}\le e^{-ns}$ for $s\in[0,1]$ and $n\in\N$. 

The next lemma will allow us to decompose integrals in a suitable way.

\begin{lemma}\label{integration}
If \, $0\le t_0\le t_1<\delta$ and $h:K\to[0,\infty]$ is a measurable function, then
$$
\int_{K_{t_0}\setminus K_{t_1}} h(x) \, dx
=\int_{\partial K}\int_{t_0}^{t_1}(1-t)^{d-1}
\langle y,u(y)\rangle h(y_t)\, dt\, \mathcal{H}^{d-1}(dy).
$$
\end{lemma}

\begin{proof}
The map $T:\partial K\times [t_0,t_1]\to K_{t_0}\setminus K_{t_1}$, $(y,t)\mapsto (1-t)y$,  
provides a bilipschitz parametrization of $K_{t_0}\setminus K_{t_1}$ with $(1-t)y=y_t\in \partial K_t$. 
The Jacobian of $T$, for $\Hde$ almost all $y\in \partial K$ and $t\in[t_0,t_1]$, is given by 
$JT(y,t)=(1-t)^{d-1}\langle y,u(y)\rangle$, 
where $u(y)$ is the ($\Hde$ a.e.) unique exterior unit normal of $\partial K$ at $y$. The assertion 
now follows from Federer's area/coarea theorem (see \cite{Federer69}). $\Box$
\end{proof}

\bigskip

\noindent
In the following, we will use the important fact that, for 
$\alpha>-1$, 
\begin{equation}
\label{inball}
\int_{\partial K}r(y)^\alpha\,\Hde(dy)<\infty,
\end{equation}
which is a  result due to 
C. Sch\"utt and  E. Werner \cite{ScW94}. 

By decomposing $\lambda$ 
in its positive and its negative part, we can henceforth assume that 
$\lambda$ is a nonnegative, integrable function.

\begin{lemma}
\label{closetoK}
As $n$ tends to infinity, 
$$
\int_{K_{n^{\frac{-1}{d+1}}}}
\PP_\varrho\left(x\not\in K_{(n)} \right)\lambda(x)\,dx=
o\left(n^{\frac{-2}{d+1}}\right).
$$
\end{lemma}

\proof Let $\delta>0$ be chosen as in Lemma \ref{genup} and the 
subsequent remark. First, 
we consider a point $x$ in $K_\delta$. Let $\omega$ be the minimal distance
between the points of $\partial K$ and $K_\delta$,
and let $z_1,\ldots,z_k$ be a maximal family of points in
$K\setminus{\rm int}(K_\delta)$ such that $\|z_i-z_j\|\geq \frac{\omega}4$
for $i\neq j$. We define $p_0>0$ by
$$
p_0:=\min\left\{\PP_\varrho\left(z_i+\tfrac{\omega}4\,B^d\right):i=1,\ldots,k\right\}.
$$
Let $x\in K_\delta$. If $x\not\in K_{(n)}$,  
then there exists a hyperplane $H(u,t)$ such that $x\in \text{int}(H^+(u,t))$ 
and $K_{(n)}\subset H^-(u,t)$. Since $x\in K_\delta$, there exists a supporting 
hyperplane $H(u,h(K_\delta,u))$ of $K_\delta$ for which $K_{(n)}\subset 
\text{int}(H^-(u,h(K_\delta,u)))$. If $z\in H(u,h(K_\delta,u))\cap\partial K_\delta$,  
then 
$$z+\frac{\omega}{2}u+\frac{\omega}{2}B^d\subset K\cap H^+(u,h(K_\delta,u)).
$$ 
By the maximality of the set $\{z_1,\ldots,z_k\}$, we have 
$$\{z_1,\ldots,z_k\}\cap 
\left( z+\frac{\omega}{2}u+\frac{\omega}{4}B^d\right)\neq\emptyset.
$$ 
Let $z_j$ lie in the intersection. 
Then $z_j+\frac{\omega}{4}B^d\subset H^+(u,h(K_\delta,u))$, and hence $x_i\notin 
z_j+\frac{\omega}{4}B^d$ for $i=1,\ldots,n$. This implies that, for $x\in K_\delta$, 
\begin{equation}\label{small0}
\PP_\varrho\left(x\not\in K_{(n)} \right)\leq k(1-p_0)^n.
\end{equation}
Put $\varepsilon:=(2(d^2-1))^{-1}$ and let $n\ge \delta^{-(d+1)}$. For  
 $y\in \partial K$ we show that 
\begin{equation}
\label{ppry}
\int_{n^{\frac{-1}{d+1}} }^{\delta}\PP_\varrho\left(y_t\not\in K_{(n)} \right)\,dt
\ll r(y)^{-\frac{d}{d+1}}n^{\frac{-2}{d+1}-\varepsilon}.
\end{equation} 
In fact, if $r(y)\leq n^{-(d+1)\varepsilon}$,  
then Lemma~\ref{genup} and (\ref{Gamma}) yield
\begin{eqnarray*}
\int_{n^{\frac{-1}{d+1}} }^{\delta}\PP_\varrho\left(y_t\not\in K_{(n)} \right)\,dt &\le &
\int_0^\delta\left(1-\gamma_1r(y)^{\frac{d-1}{2}}t^{\frac{d+1}{2}}\right)^n\, dt\\
&\ll& r(y)^{-\frac{d-1}{d+1}}n^{-\frac{2}{d+1}}\\
&\le& 
r(y)^{-\frac{d}{d+1}}n^{-\frac{2}{d+1}-\varepsilon},
\end{eqnarray*}
where the assumption on $r(y)$ is used for the last estimate.

If $r(y)\geq n^{-(d+1)\varepsilon}$
and $n\geq n_0$, where $n_0$ depends
on $K$, $\varrho$ and $\lambda$, then
Lemma~\ref{genup} implies for all
$t\in(n^{\frac{-1}{d+1}},\delta)$ that 
$$
\PP_\varrho\left(y_t\not\in K_{(n)} \right)\ll
\left(1-\gamma_1n^{-\frac{d^2-1}2\,\varepsilon-\frac{1}2}\right)^n=(1-\gamma_1n^{-3/4})^n
\le e^{-\gamma_1n^{1/4}}\le r(K)^{-\frac{d}{d+1}}n^{\frac{-2}{d+1}-\varepsilon},
$$
which again yields (\ref{ppry}).
In particular, writing $I$ to denote
the integral in Lemma~\ref{closetoK}, we obtain from  
Lemma \ref{integration}, \eqref{small0},  \eqref{ppry} and (\ref{inball}) that
\begin{eqnarray*}
I&\ll&
\int_{K_{\delta}}
\PP_\varrho\left(x\not\in K_{(n)} \right)\lambda(x)\,dx+
\int_{\partial K}\int_{n^{\frac{-1}{d+1}}}^{\delta}
\PP_\varrho\left(y_t\not\in K_{(n)} \right)\,dt\,\Hde(dy)\\
&\ll&k(1-p)^n+\int_{\partial K}
r(y)^{-\frac{d}{d+1}}n^{\frac{-2}{d+1}-\varepsilon}\,\Hde(dy)
\ll n^{\frac{-2}{d+1}-\varepsilon},
\end{eqnarray*}
where we also used that $\lambda$ is integrable on $K$ and 
bounded on $K\setminus K_\delta$. This is the required estimate. 
\proofbox

\bigskip

\noindent
It follows from  \eqref{basic1}, Lemma~\ref{closetoK} and Lemma \ref{integration} 
 that
\begin{align*}
&\lim_{n\to\infty}n^{\frac2{d+1}}\, \E_\varrho
\int_{K\setminus K_{(n)}}\lambda(x)\,dx\\
  &\qquad = \lim_{n\to\infty}n^{\frac2{d+1}} 
\int_{K}\PP_\varrho\left(x\not\in K_{(n)} \right)\lambda(x)\,dx\\
&\qquad =\lim_{n\to\infty}\int_{\partial K}\int_0^{n^{\frac{-1}{d+1}}}
n^{\frac2{d+1}} (1-t)^{d-1}\langle y,u(y)\rangle\PP_\varrho\left(y_t\not\in K_{(n)} \right)
\lambda(y_t)\,dt\,\Hde(dy).
\end{align*}
Lemma~\ref{genup} and (\ref{Gamma}) imply that if
$y\in\partial K$ and $r(y)>0$, then
$$
\int_0^{n^{\frac{-1}{d+1}}}
n^{\frac2{d+1}} \PP_\varrho\left(y_t\not\in K_{(n)} \right)
\langle y,u(y)\rangle\lambda(y_t)\,dt\ll r(y)^{-\frac{d-1}{d+1}}.
$$
Therefore, by (\ref{inball}) and since $\lambda$ is bounded and continuous in a 
neighborhood of $\partial K$ we may apply Lebesgue's dominated convergence theorem, 
and thus we conclude
\begin{equation}
\label{limitform}
\lim_{n\to\infty}n^{\frac2{d+1}}\E_\varrho
\int_{K\setminus K_{(n)}}\lambda(x)\,dx=
\int_{\partial K}\lambda(y)J_\varrho(y)\,\Hde(dy),
\end{equation}
where
$$
J_\varrho(y):=\lim_{n\to\infty}
\int_0^{n^{\frac{-1}{d+1}}}n^{\frac2{d+1}}\langle y,u(y)
\rangle\PP_\varrho\left(y_t\not\in K_{(n)} \right)\,dt,
$$
for $\Hde$ almost all $y\in\partial K$.

\begin{lemma}
\label{zerocurv}
If $y\in\partial K$ is a normal boundary point of $K$ with $\kappa(y)=0$, 
then $J_\varrho(y)=0$.
\end{lemma}
\proof In view of the estimate \eqref{except}, it is sufficient to prove that for 
any given $\varepsilon>0$, 
\begin{equation}
\label{closeeps}
\int_0^{n^{\frac{-1}{d+1}}}n^{\frac{2}{d+1}}
\PP_\varrho\left(y_t\not\in K_{(n)},o\in K_{(n)} \right)\,dt \ll\varepsilon, 
\end{equation}
if $n$ is sufficiently large. 
We choose the coordinate axes in $u(y)^\bot$
parallel to the principal curvature directions of $K$ at $y$, and denote by 
$\Theta'_1,\ldots,\Theta'_{2^{d-1}}$ the corresponding coordinate corners.
For $i=1,\ldots,2^{d-1}$ and $t\in(0,n^{\frac{-1}{d+1}})$, let
$$
\Theta_{i,t}:=C(y,t)\cap\left(y_t+\left[\Theta'_i,\R^+y\right]
\right),
$$
and hence, if $n$ is large enough, then
$$
\int_{\Theta_{i,t}}\varrho(x)\,dx\gg V(\Theta_{i,t}),
$$
since $\varrho$ is continuous and positive near $\partial K$. 
If $y_t\not\in K_n$ and $o\in K_{(n)}$, then there exists a 
halfspace $H^-$ which contains $K_{(n)}$ and for which $y_t\in\partial H^-$. Moreover, 
for some $i\in\{1,\ldots,2^{d-1}\}$ the interior of $H^-$ is 
disjoint from $\Theta_{i,t}$. Hence, as in the proof of Lemma \ref{genup},
\begin{equation}
\label{zerocurvetheta}
\PP_\varrho\left(y_t\not\in K_{(n)},o\in K_{(n)} \right)\ll
\sum_{i=1}^{2^{d-1}}\left(1-\gamma_6 V(\Theta_{i,t})\right)^n.  
\end{equation}
Since $\partial K$ is twice differentiable in the generalized sense
at $y$, we have $r(y)>0$. By assumption, $\kappa(y)=0$, therefore  
one principal curvature at $y$ is zero, and hence
less than $\varepsilon^{d+1}r(y)^{d-2}$. In particular, 
there exists $\delta'\in(0,\delta)$, which by (\ref{yu(y)}) depends 
only on $y$ and $\varepsilon$, such that
if $i\in\{1,\ldots,2^{d-1}\}$ and $t\in(0,\delta')$, then
$$
{\cal H}^{d-1}\left((y_t+\Theta'_i)\cap K\right)
\gg \sqrt{t\varepsilon^{-(d+1)}r(y)^{-(d-2)}}
\cdot\sqrt{tr(y)}^{d-2}.
$$
We deduce $V(\Theta_{i,t})\gg \varepsilon^{-\frac{d+1}2}t^{\frac{d+1}2}$. 
Therefore (\ref{closeeps}) follows from (\ref{Gamma})
and (\ref{zerocurvetheta}).
\proofbox

Next we consider the case of a normal boundary point $y\in\partial K$ with 
$\kappa(y)>0$. 
First, we prove that $J_\varrho(y)$ depends
only on the random points near $y$
(see Lemma~\ref{pointsneary}). In a second step, we
compare the simplified expression obtained for $J_\varrho(y)$ with 
the corresponding expression which is obtained if $K$ is a ball. 

We start by reparametrizing $y_t$ in terms of the
probability measure of the corresponding cap.
 For $t\in(0,n^{\frac{-1}{d+1}})$, where $n\ge n_0$ is sufficiently large so that 
 $\varrho$ is positive and continuous on $C(y,t)$, for all $y\in\partial K$, we put
$$
\tilde{y}_s:=y_t
$$
where for given $s>0$ (sufficiently small) the corresponding $t=t(s)$ is determined by the relation
\begin{equation}\label{defst}
s=\int_{C(y,t)}\varrho(x)\,dx. 
\end{equation}
It is easy to see that the right-hand side of \eqref{defst} is a continuous 
and strictly increasing function $s=s(t)$ of $t$, if $t>0$ is sufficiently small. This 
implies that for a given $s>0$ (sufficiently small) there is a unique $t(s)$ such that \eqref{defst} 
is satisfied. 

Moreover, observe that 
\begin{equation}\label{dustar}
\frac{ds}{dt}=\langle u(y),y\rangle\int_{H(y,t)\cap K}\varrho(x)\, \Hde(dx)
\end{equation}
for $t\in (0,n^{\frac{-1}{d+1}})$. 
We further define  
$$\widetilde{C}(y,s):=C(y,t)\qquad\text{and}\qquad 
\widetilde{H}(y,s):=\{x\in \R^d:\langle u(y),x\rangle=
\langle u(y),\tilde{y}_s\rangle\},
$$ 
where $t=t(s)$.

Let $Q$ denote the second fundamental form of $\partial K$ at $y$ (cf.\ \eqref{diff}), 
considered as a function on $u(y)^\perp$. We define 
$$
E:=\{z\in u(y)^\perp:Q(z)\le 1\}.
$$
and put $u:=u(y)$. Choosing a suitable 
orthonormal basis $v_1,\ldots,v_{d-1}$ of $u(y)^\bot$, 
we have 
$$
Q(z)=\sum_{i=1}^{d-1}k_i(y)z_i^2,
$$ where $k_i(y)$, $i=1,\ldots,d-1$, are the generalized principal curvatures 
of $K$ at $y$ and where $z=z_1v_1+\ldots +z_{d-1}v_{d-1}$. Since $y$ is a 
normal boundary point of $K$, there is a nondecreasing function $\mu:(0,\infty)\to \R$ with 
$\lim_{r\to 0^+}\mu(t)=1$ such that 
\begin{equation}\label{section}
\frac{\mu(r)^{-1}}{\sqrt{2r}}(K\cap H(u,h(K,u)-r)+ru-y)\subset E\subset 
\frac{\mu(r)}{\sqrt{2r}}(K\cap H(u,h(K,u)-r)+ru-y).
\end{equation}
In the following, $\mu_i:(0,\infty)\to \R$, $i=1,2,\ldots$, always denote nondecreasing functions 
with $\lim_{r\to 0^+}\mu(t)=1$. Applying \eqref{section} and Fubini's theorem, we get
$$
V(K\cap H^+(u,h(K,u)-r))=\mu_1(r)\frac{(2r)^{\frac{d+1}{2}}}{d+1}
\alpha_{d-1}\kappa(y)^{-\frac{1}{2}},
$$
which yields that
\begin{equation}\label{relst}
s(t)=\mu_2(t)\frac{(2t\langle y,u\rangle)^{\frac{d+1}{2}}}{d+1}
\alpha_{d-1}\kappa(y)^{-\frac{1}{2}}\varrho(y),
\end{equation}
since $\varrho$ is continuous at $y$. Moreover, defining
$$
\eta:=(d+1)^{\frac1{d+1}}\alpha_{d-1}^{-\frac{1}{d+1}}\varrho(y)^{\frac{-1}{d+1}}
\kappa(y)^{\frac1{2(d+1)}},
$$
we obtain
\begin{equation}
\label{asymparea}
\lim_{s\to 0^+}s^{\frac{-1}{d+1}}[(\widetilde{H}(y,s)\cap K)-\tilde{y}_s]
=\eta\cdot E
\end{equation}
in the sense of the Hausdorff metric on compact convex sets
(see Schneider \cite{Sch93} or  Gruber \cite{Gru07}). Here we also use that 
\begin{equation}\label{vertical}
\lim_{s\to 0^+}s^{-\frac{1}{d+1}}(\tilde{y}_s-\langle \tilde{y}_s,u\rangle u)=o. 
\end{equation}
Now it follows from \eqref{dustar} and \eqref{asymparea} that (\ref{limitform}) turns into
$$
J_\varrho(y)=
(d+1)^{-\frac{d-1}{d+1}}\alpha_{d-1}^{-\frac{2}{d+1}}\varrho(y)^{\frac{-2}{d+1}} \kappa(y)^{\frac1{d+1}}
 \lim_{n\to\infty}\int_0^{g(n,y)}n^{\frac2{d+1}}
\PP_\varrho\left(\tilde{y}_s\not\in K_{(n)} \right)s^{-\frac{d-1}{d+1}}\,ds,
$$
where
$$
\lim_{n\to\infty}n^{\frac12}g(n,y)=
(d+1)^{-1}\alpha_{d-1}{\varrho(y)(2\langle u(y),y\rangle)^{\frac{d+1}2} }
\kappa(y)^{-\frac12} .
$$

The rest of the proof is devoted to identifying the asymptotic behavior of the integral. 
First, we adjust the domain of integration and the integrand in a suitable way. In a second step, 
the resulting expression is compared to the case where $K$ is the unit ball. 
We recall that $x_1,\ldots,x_n$ are random points in $K$, 
and we put $\Xi_n:=\{x_1,\ldots,x_n\}$, and hence $K_{(n)}=[\Xi_n]$. Let
$\#X$ denote the cardinality of a finite set $X\subset\R^d$.

\begin{lemma}
\label{pointsneary}
For $\varepsilon\in(0,1)$, there exist
$\alpha,\beta>1$ and an integer $k>1$, depending only on $\varepsilon$
and $d$, with the following property.
If $y\in\partial K$ is a normal boundary point of $K$ with $\kappa(y)>0$ and if $n>n_0$, 
where $n_0$ depends on $\varepsilon,y,K,\varrho$,
then
$$
\int_0^{g(n,y)}
\PP_\varrho\left(\tilde{y}_s\not\in K_{(n)} \right)s^{-\frac{d-1}{d+1}}\,ds
=
\int_{\frac{\varepsilon^{(d+1)/2}}n}^{\frac{\alpha}n}
\varphi(K,y,\varrho,\varepsilon,s)s^{-\frac{d-1}{d+1}}\,ds
+O\left(\frac{\varepsilon}{n^{\frac2{d+1}}}\right),
$$
where
$$
\varphi(K,y,\varrho,\varepsilon,s)=
\PP_\varrho\left(\left(\tilde{y}_s\not\in 
[\widetilde{C}(y,\beta s)\cap\Xi_n]\right)
\mbox{\rm and } 
\left(\#(\widetilde{C}(y,\beta s)\cap\Xi_n)\leq k\right)\right).
$$
\end{lemma}
\proof Let $Q$ be the second fundamental form of $\partial K$ at the normal 
boundary point $y$, and let
$v_1,\ldots,v_{d-1}$ be an orthonormal basis of $u(y)^\bot$ with
 respect to $Q$, as described above. Let 
$\Theta'_1,\ldots,\Theta'_{2^{d-1}}$ be the corresponding coordinate corners,
and, for $i=1,\ldots,2^{d-1}$ and for $s\in(0,n^{-1/2})$, put 
$$
\widetilde{\Theta}_{i,s}:=
\widetilde{C}(y,s)\cap\left(\tilde{y}_s+\left[\Theta'_i,\R^+y\right]
\right).
$$
Let $A_s$, $s>0$, be the affine map of $\R^d$ with $A_s(y)=y$ for which the associated linear map 
$\widetilde{A}_s$ is determined by $\widetilde{A}_s(v)=s^{\frac{1}{d+1}}v$, for $v\in u^\perp$, and 
$\widetilde{A}_s(u)=s^{\frac{2}{d+1}}u$. Then $\det(\widetilde{A}_s)=s$ and 
${A}_{s^{-1}}(\widetilde{C}(y,s))$ converges in the Hausdorff metric as $s\to 0^+$ to 
the cap $\widetilde{C}(y)$ of the osculating paraboloid of $K$ at $y$ having volume $\varrho(y)^{-1}$. 
Here we use that $\varrho$ is continuous at $y$, $\varrho(y)>0$ and relation \eqref{defst}. Let $\lambda>0$ be 
such that $\tilde{y}:=y-\lambda u\in \partial \widetilde{C}(y)$. Then $A_{s^{-1}}(\widetilde{\Theta}_{i,s})$ 
converges in the Hausdorff metric as $s\to 0^+$ to $\widetilde{C}(y)\cap(\tilde{y}+[\Theta'_i,\R^+u])$, 
since \eqref{vertical} is satisfied. Using again that $\varrho$ is continuous and positive at $y$, 
we deduce that
\begin{eqnarray*}
\lim_{s\to 0^+}s^{-1}\int_{\widetilde{\Theta}_{i,s}}\varrho(x)\,dx&=&
\lim_{s\to 0^+}s^{-1}V(\widetilde{\Theta}_{i,s})\varrho(y)\\
&=&\lim_{s\to 0^+}V(A_{s^{-1}}(\widetilde{\Theta}_{i,s}))\varrho(y)\\
&=&V(\widetilde{C}(y)\cap(\tilde{y}+[\Theta'_i,\R^+u]))\varrho(y)\\
&=&2^{-(d-1)}V(\widetilde{C}(y))\varrho(y)\\
&=&2^{-(d-1)}\lim_{s\to 0^+}V(A_{s^{-1}}(\widetilde{C}(y,s))\varrho(y)\\
&=&2^{-(d-1)}\lim_{s\to 0^+}s^{-1}V(\widetilde{C}(y,s))\varrho(y)\\
&=&2^{-(d-1)}\lim_{s\to 0^+}s^{-1}\int_{\widetilde{C}(y,s)}\varrho(x)\,dx\\
&=&2^{-(d-1)},
\end{eqnarray*}
that is
\begin{equation}
\label{thetavol}
\lim_{s\to 0^+}s^{-1}\int_{\widetilde{\Theta}_{i,s}}\varrho(x)\,dx=2^{-(d-1)}.
\end{equation}
Let $\alpha>1$ be chosen such that
$$
2^{d-1+2d/(d+1)}\int_{2^{-d}\alpha}^\infty e^{-x}x^{\frac{2}{d+1}-1}\, dx \le \varepsilon.
$$
Then we first choose $\beta\ge (16(d-1))^{d+1}$ such that 
\begin{align*}
2^{d-1}e^{-d^{-1}2^{-(d+3)}\beta^{\frac1{d+1}}\varepsilon^{\frac{d+1}2}}& \le 
\frac{\varepsilon}{\alpha^{\frac2{d+1}} } , \\
\intertext{and then we fix an integer $k>1$ such that}
\frac{(\alpha\beta)^k}{k!}&\le \frac{\varepsilon}{\alpha^{\frac2{d+1}} }.
\end{align*}
Lemma~\ref{pointsneary} follows from the following three statements,
which we will prove assuming that $n$ is sufficiently large.
\begin{description}
\item{(i)} 
$$\int_0^{g(n,y)}\PP_\varrho\left(\tilde{y}_s\not\in K_{(n)} \right)s^{-\frac{d-1}{d+1}}\,ds
 =\int_{\frac{\varepsilon^{(d+1)/2}}n}^{\frac{\alpha}n}
\PP_\varrho\left(\tilde{y}_s\not\in K_{(n)} \right)s^{-\frac{d-1}{d+1}}\,ds
+O\left(\frac{\varepsilon}{n^{\frac2{d+1}}}\right).
$$
\item{(ii)} If $\frac{\varepsilon^{(d+1)/2}}n<s<\frac{\alpha}n$, then
$$
\PP_\varrho\left(\#\left(\widetilde{C}(y,\beta s)\cap\Xi_n\right)\geq k \right)
=O\left(\frac{\varepsilon}{\alpha^{\frac2{d+1}}}\right).
$$
\item{(iii)} If $\frac{\varepsilon^{\frac{d+1}2}}n<s<\frac{\alpha}n$, then
$$
\PP_\varrho\left(\tilde{y}_s\not\in K_{(n)} \right)=
\PP_\varrho\left(\tilde{y}_s\not\in
\left[\widetilde{C}(y,\beta s)\cap\Xi_n\right] \right)
+O\left(\frac{\varepsilon}{\alpha^{\frac2{d+1}}}\right).
$$
\end{description}

\noindent
To prove (i), we first observe that
$$
\int_0^{\frac{\varepsilon^{(d+1)/2}}n}
\PP_\varrho\left(\tilde{y}_s\not\in K_{(n)} \right)s^{-\frac{d-1}{d+1}}\,ds\leq
\int_0^{\frac{\varepsilon^{(d+1)/2}}n}s^{-\frac{d-1}{d+1}}\,ds\ll
\frac{\varepsilon}{n^{\frac2{d+1}}}.
$$
If $\frac{\alpha}n<s<g(n,y)$, $o\in K_{(n)}$, $\tilde{y}_s\not\in K_{(n)}$ and if $n$ is 
sufficiently large, then there is some $i\in\{1,\ldots,2^{d-1}\}$ such that 
$\widetilde{\Theta}_{i,s}\cap K_{(n)}=\emptyset$, 
and hence \eqref{except} and (\ref{thetavol}) yield 
\begin{equation}\label{pestimate}
\PP_\varrho\left(\tilde{y}_s\not\in K_{(n)} \right)\ll 
2^{d-1}(1-2^{-d}s)^n\le 2^{d-1}e^{-2^{-d}ns}.
\end{equation}
Therefore, by the definition of $\alpha$, we get
\begin{align*}
\int_{\frac{\alpha}{n}}^{g(n,y)}
\PP_\varrho\left(\tilde{y}_s\not\in K_{(n)} \right)s^{-\frac{d-1}{d+1}}\,ds
& \ll 2^{d-1}\int_{\frac{\alpha}{n}}^\infty e^{-2^{-d}ns}s^{\frac{2}{d+1}-1}\, ds\\
& = 2^{d-1}2^{2d/(d+1)}n^{-\frac{2}{d+1}}\int_{2^{-d}\alpha}^\infty e^{-x}
x^{\frac{2}{d+1}-1}\, dx\\
&\le \varepsilon\, n^{-\frac{2}{d+1}},
\end{align*}
which verifies (i).

Next (ii) simply follows from \eqref{defst} as if $s< \frac{\alpha}n$, then
$$
\PP_\varrho\left(\#\left(\widetilde{C}(y,\beta s)\cap
\Xi_n\right)\geq k \right)=\binom{n}{k}(\beta s)^k 
\le \binom{n}{k}\left(\frac{\alpha\beta}n\right)^k<
\frac{(\alpha\beta)^k}{k!}\le \frac{\varepsilon}{\alpha^{\frac2{d+1}}}.
$$

Now we prove (iii). To this end, for $s$ in the given range, our plan is to construct
sets $\widetilde{\Omega}_{1,s},\ldots,\widetilde{\Omega}_{2^{d-1},s}\subset K$ such that 
\begin{equation}
\label{Omegavol}
\int_{\widetilde{\Omega}_{i,s}}\varrho(x)\,dx\geq d^{-1}2^{-(d+3)}\beta^{\frac1{d+1}}s,\qquad
\mbox{  for \ }i=1,\ldots,2^{d-1},
\end{equation}
and if $\tilde{y}_s\in K_{(n)}$
but $\tilde{y}_s\not\in\left[\widetilde{C}(y,\beta s)\cap\Xi_n\right]$, 
then  $\Xi_n\cap\widetilde{\Omega}_{i,s}=\emptyset$ for some $i\in \{1,\ldots,2^{d-1}\}$. 

For $i=1,\ldots,2^{d-1}$, let $w_i\in \Theta'_i$ be the vector
whose coordinates (up to sign) in the basis $v_1,\ldots,v_{d-1}$ are
$$
w_i:=\left(\sqrt{\beta}s\right)^{\frac{1}{d+1}}\frac{\eta}{2\sqrt{d-1}}\left(\pm  
\frac{1}{\sqrt{k_1(y)}},\ldots, \pm\frac{1}{\sqrt{k_{d-1}(y)}}\right).
$$
Further, for $i=1,\ldots,2^{d-1}$ we define
$$
\widetilde{\Omega}_{i,s}=
[\tilde{y}_{\sqrt{\beta}\, s}+w_i,K\cap(\tilde{y}_s+\Theta'_i)].
$$
Then, if $s>0$ is small enough,  $\tilde{y}_{\sqrt{\beta}\, s}+w_i\in K$, and 
hence $\widetilde{\Omega}_{i,s}\subset K$. Here we use that 
$$
w_i\in(\sqrt{\beta}s)^{\frac{1}{d+1}}\frac{1}{2}\eta E
$$
and therefore by \eqref{asymparea}
$$
\tilde{y}_{\sqrt{\beta}s}+w_i\in\widetilde{H}(y,\sqrt{\beta}s)\cap K\subset K.
$$
Using that $\tilde{y}_s=(1-t)y$, where $s$ and $t$ are related by \eqref{relst}, 
and if $s,t>0$ are sufficiently small, we obtain
\begin{equation}\label{A}
\langle u(y),\tilde{y}_s-\tilde{y}_{\sqrt{\beta}\, s}\rangle
>\frac{\beta^{\frac1{d+1}}-1}2\,\langle u(y),y-\tilde{y}_s\rangle
>\frac{\beta^{\frac1{d+1}} }4\langle u(y),y-\tilde{y}_s\rangle,
\end{equation}
since $\beta\ge 2^{d+1}$. Moreover, we have
\begin{equation}\label{B}
\langle u(y),y-\tilde{y}_s\rangle
\cdot {\cal H}^{d-1}\left(K\cap(\tilde{y}_s+\Theta'_i)\right)\ge V(\widetilde{\Theta}_{i,s}).
\end{equation}
Combining \eqref{A}, \eqref{B},  (\ref{thetavol}) and the continuity of
$\varrho$ at $y$ with $\varrho(y)>0$, we deduce \eqref{Omegavol}, that is
\begin{eqnarray*}
\int_{\widetilde{\Omega}_{i,s}}\varrho(x)\,dx &\ge&
\frac{1}{\sqrt{2}}\frac{1}{d}\varrho(y)\langle u(y),\tilde{y}_s-\tilde{y}_{\sqrt{\beta}\, s}\rangle
{\cal H}^{d-1}\left(K\cap(\tilde{y}_s+\Theta'_i)\right)\\
&\ge&\frac{\beta^{\frac1{d+1}} }{4}\frac{1}{\sqrt{2}d}V(\widetilde{\Theta}_{i,s})\\
&\ge&\frac{\beta^{\frac1{d+1}} }{4}\frac{1}{{2}d}
\int_{\widetilde{\Theta}_{i,s}}\varrho(x)\,dx\\
&\ge &\frac{\beta^{\frac1{d+1}}s}{8d\,2^d}.
\end{eqnarray*}

It is still left to prove that
if $\tilde{y}_s\in K_{(n)}$
but $\tilde{y}_s\not\in \left[\widetilde{C}(y,\beta s)\cap\Xi_n\right]$, 
then  $\Xi_n\cap\widetilde{\Omega}_{i,s}=\emptyset$ for some $i\in\{1,\ldots,2^{d-1}\}$.
So we assume that $\tilde{y}_s\in K_{(n)}$
but $\tilde{y}_s\not\in \left[\widetilde{C}(y,\beta s)\cap\Xi_n\right]$. Then there exist 
$a\in \left[\widetilde{C}(y,\beta s)\cap\Xi_n\right]$
and $b\in K_{(n)}\setminus \widetilde{C}(y,\beta s)$ such that
$\tilde{y}_s\in [a,b]$, and hence 
there exists a hyperplane $H$ containing $\tilde{y}_s$ 
bounding the halfspaces $H^+$ and $H^-$ such that
$\widetilde{C}(y,\beta s)\cap\Xi_n\subset {\rm int}(H^+)$ and
$b\in \text{int}(H^-)$.

Next we show that there exists $q\in [\tilde{y}_s,b]$ such that 
\begin{equation}\label{stern2}
q\in H^-\cap \left(\tilde{y}_{\sqrt{\beta}\, s}+
\frac{\eta}{2\sqrt{d-1}}(\sqrt{\beta}s)^{\frac{1}{d+1}} E\right).
\end{equation} 
In fact, define $q:=[\tilde{y}_s,b]\cap \widetilde{H}(y,\sqrt{\beta} s)$ and 
$q':=[\tilde{y}_s,b]\cap \widetilde{H}(y,{\beta} s)$. Since $a\in H^+$ and $\tilde{y}_s\in H$, 
it follows that $q\in H^-$. 
From (\ref{asymparea}) we get 
\begin{equation}\label{include}
\widetilde{H}(y,\beta s)\cap K\subset 
\tilde{y}_{\beta s}+2\beta^{\frac1{d+1}}s^{\frac{1}{d+1}}\eta E.
\end{equation}
Applying \eqref{relst}, we deduce
\begin{eqnarray}
\langle u(y),\tilde{y}_s-\tilde{y}_{\beta\, s}\rangle
&<&  \frac{\beta^{\frac{2}{d+1}}}{\beta^{\frac{2}{d+1}}-1} 
 \cdot\frac{\beta^{\frac2{d+1}}-1}{\beta^{\frac2{d+1}}-\beta^{\frac{1}{d+1}}}
\langle u(y),\tilde{y}_{\sqrt{\beta}\, s}-\tilde{y}_{\beta s}\rangle\nonumber\\
&<&\frac{\beta^{\frac{1}{d+1}}}{\beta^{\frac{1}{d+1}}-1} 
\langle u(y),\tilde{y}_{\sqrt{\beta}\, s}-\tilde{y}_{\beta s}\rangle.\label{estimbeta}
\end{eqnarray}
Furthermore, elementary geometry yields
$$
\frac{\|q-\tilde{y}_{\sqrt{\beta}s}\|}{\|q'-\tilde{y}_{\beta s}\|}=\frac{\langle u,\tilde{y}_s-
\tilde{y}_{\sqrt{\beta}s}\rangle}{\langle u,\tilde{y}_s-\tilde{y}_{\beta s}\rangle}.
$$
Then \eqref{include} and \eqref{estimbeta} imply that
\begin{eqnarray*}
q&\in&\tilde{y}_{\sqrt{\beta}s}+\frac{\langle u,\tilde{y}_s-
\tilde{y}_{\sqrt{\beta}s}\rangle}{\langle u,\tilde{y}_s-\tilde{y}_{\beta s}\rangle}
\cdot 2(\beta s)^{\frac{1}{d+1}}\eta E\\
&\subset&\tilde{y}_{\sqrt{\beta}s}+\left(1-\frac{\langle u,\tilde{y}_{\sqrt{\beta}s}-
\tilde{y}_{{\beta}s}\rangle}{\langle u,\tilde{y}_s-\tilde{y}_{\beta s}\rangle}\right)
\cdot 2\beta^{\frac{1}{d+1}}s^{\frac{1}{d+1}}\eta E\\
&\subset&\tilde{y}_{\sqrt{\beta}s}+2s^{\frac{1}{d+1}}\eta E\\
&\subset& \tilde{y}_{\sqrt{\beta}s}+\frac{1}{2\sqrt{d-1}}(\sqrt{\beta}s)^{\frac{1}{d+1}}\eta E,
\end{eqnarray*}
where $\beta\ge (16(d-1))^{d+1}$ is used for the last inclusion. 
 Now there exists some $i\in\{1,\ldots,2^{d-1}\}$ such that $\tilde{y}_s+\Theta_i'\subset H^-$,
and hence  $q+\Theta_i'\subset H^-$. By \eqref{stern2} this finally yields 
$$\tilde{y}_{\sqrt{\beta}s}+w_i \subset q+\Theta'_i\subset H^-.
$$
Therefore we obtain $\widetilde{\Omega}_{i,s}\cap \Xi_n=\emptyset$.

Finally, (iii) follows as
if $\frac{\varepsilon^{\frac{d+1}2}}n<s<\frac{\alpha}n$, then
\begin{eqnarray*}
&&0\le \PP_\varrho\left(\tilde{y}_s\not\in\left[\widetilde{C}(y,\beta s)\cap
\Xi_n\right] \right)
-\PP_\varrho\left(\tilde{y}_s\not\in K_{(n)} \right)\\
& &\qquad \qquad\leq
\sum_{i=1}^{2^{d-1}}(1-\int_{\widetilde{\Omega}_{i,s}}\varrho(x)\,dx)^n\\
&&\qquad \qquad\le \sum_{i=1}^{2^{d-1}}e^{-n\int_{\widetilde{\Omega}_{i,s}}\varrho(x)\,dx}\\
&&\qquad \qquad\le 2^{d-1}e^{-d^{-1}2^{-(d+3)}{\beta}^{\frac{1}{d+1}}\,\varepsilon^{\frac{d+1}2}}\\
&&\qquad \qquad\le {\varepsilon}\,{\alpha^{-\frac2{d+1}} },
\end{eqnarray*}
by the choice of $\beta$.
\mbox{ }\proofbox

\bigskip

\noindent
{\bf Remark} As a consequence of the proof of Lemma \ref{pointsneary}, it follows that 
\begin{equation}
\label{limitform0}
J_\varrho(y)=
(d+1)^{-\frac{d-1}{d+1}}\alpha_{d-1}^{-\frac{2}{d+1}}\varrho(y)^{\frac{-2}{d+1}}
 \kappa(y)^{\frac1{d+1}} 
\lim_{n\to\infty}\int_0^{n^{-1/2}}n^{\frac2{d+1}}
\PP_\varrho\left(\tilde{y}_s\not\in K_{(n)} \right)s^{-\frac{d-1}{d+1}}\,ds.
\end{equation}
In fact, since $g(n,y)\ll n^{-1/2}$, it is sufficient to show that
$$
\lim_{n\to\infty}n^{\frac{2}{d+1}}\int_{c_1n^{-1/2}}^{c_2n^{-1/2}}
\PP_\varrho\left(\tilde{y}_s\not\in K_{(n)} \right)s^{-\frac{d-1}{d+1}}\,ds=0
$$
for any two constants $0<c_1\le c_2<\infty$. Since the estimate \eqref{pestimate} 
can be applied, we get
\begin{eqnarray*}
n^{\frac{2}{d+1}}\int_{c_1n^{-1/2}}^{c_2n^{-1/2}}
\PP_\varrho\left(\tilde{y}_s\not\in K_{(n)} \right)s^{-\frac{d-1}{d+1}}\,ds
&\ll& n^{\frac{2}{d+1}}\int_{c_1n^{-1/2}}^{c_2n^{-1/2}}
e^{-2^{-d}ns}s^{\frac{2}{d+1}-1}\,ds \\
&\ll& \int_{2^{-d}c_1n^{1/2}}^{2^{-d}c_2n^{1/2}}
e^{-r}r^{\frac{2}{d+1}-1}\,dr ,
\end{eqnarray*}
from which the conclusion follows.

\bigskip

Subsequently, we write ${\bf 1}$ to denote the constant
one function on $\R^d$.
For the unit ball $B^d$, we recall that $B^d_{(n)}$
denotes the convex hull of $n$ random points 
distributed uniformly and independently 
in $B^d$. We fix a point $w\in\partial B^d$, and for $s\in (0,\frac12)$,
define $\tilde{w}_s:=t\cdot w$, where $t\in (0,1)$ is chosen such that
$$
s=\alpha_d^{-1}\cdot V(\{x\in B^d:\,\langle x,w\rangle\geq \langle \tilde{w}_s,w\rangle\}).
$$
A classical result due to  J.A. Wieacker \cite{Wie78}
is that 
$$
\lim_{n\to\infty}n^{\frac2{d+1}}\E_{{\bf 1},B^d}
V(B^d\setminus B^d_{(n)})=c_d\,\omega_d\,\alpha_d^{\frac{2}{d+1}},
$$
where the constant $c_d$ is given in (\ref{wieackercons}).
It follows from (\ref{limitform}), (\ref{limitform0}) and the preceding remark that
\begin{equation}
\label{ball}
\lim_{n\to\infty}\int_0^{n^{-1/2}}n^{\frac2{d+1}} 
\PP_{{\bf 1},B^d}\left(\tilde{w}_s\not\in B^d_{(n)} \right)
s^{-\frac{d-1}{d+1}}\,ds 
=c_d\,(d+1)^{\frac{d-1}{d+1}}\alpha_{d-1}^{\frac{2}{d+1}}.
\end{equation}
We are now going to show that the same limit is obtained if $B^d$ is 
replaced by the convex body $K$ and if a normal boundary point $y$ of $K$ 
with positive Gauss curvature is considered instead of $w\in \partial B^d$.

\begin{lemma}
\label{compareball}
If $y\in\partial K$ is a normal boundary point of $K$ satisfying $\kappa(y)>0$, then
$$
\lim_{n\to\infty}\int_0^{n^{-1/2}}n^{\frac2{d+1}}
\PP_\varrho\left(\tilde{y}_s\not\in K_{(n)} \right)s^{-\frac{d-1}{d+1}}\,ds
=c_d(d+1)^{\frac{d-1}{d+1}}\alpha_{d-1}^{\frac{2}{d+1}}.
$$
\end{lemma}
{\it Proof: } Let $\varepsilon\in(0,1)$ be arbitrarily chosen. 
According to Lemma~\ref{pointsneary} and its notation and by the preceding remark, 
if $n$ is sufficiently large, we have
\begin{eqnarray}
\int_0^{n^{-1/2}}
\PP_\varrho\left(\tilde{y}_s\not\in K_{(n)} \right)s^{-\frac{d-1}{d+1}}\,ds&=&
O\left(\frac{\varepsilon}{n^{\frac2{d+1}}}\right)+\sum_{i=0}^k
\binom{n }{ i}\int_{\frac{\varepsilon^{(d+1)/2}}n}^{\frac{\alpha}n}
(\beta s)^i(1-\beta s)^{n-i}\nonumber\\
&&\qquad \times\,
\PP_{\varrho,\widetilde{C}(y,\beta s)}
\left(\tilde{y}_s\not\in \widetilde{C}(y,\beta s)_{(i)} \right)
s^{-\frac{d-1}{d+1}}\,ds.\label{finapprox}
\end{eqnarray}

We fix a unit vector $p$, 
and consider the reference paraboloid $\Psi$ which is the graph
of $z\mapsto \|z\|^2$ on $p^\bot$. For $\tau>0$, define
$$
C(\tau):=\left\{z+tp:z\in p^{\bot}\mbox{ \ and \ }
\|z\|^2\leq t\leq \tau^{\frac{2}{d+1}}\right\},
$$
that is a cap of $\Psi$  of height $\tau^{\frac{2}{d+1}}$. It is easy to check that 
$V(C(\tau))=\tau V(C(1))$. We define 
$$
\tilde{s}(\beta,s):=\frac{V(\widetilde{C}(y,\beta s))}{V(C(\beta))}.
$$
Then \eqref{defst} implies that
$$
\tilde{s}(\beta,s)=\frac{\beta s}{\mu(\beta,s)\varrho(y)\beta V(C(1))}=
\frac{s}{\mu(\beta,s)\varrho(y) V(C(1))},
$$
where $\mu(\beta,s)\to 1$ as $s\to 0^+$. Let $A_s$, $s>0$, denote the affinity of $\R^d$ with $A_s(y)=y$ 
for which the associated linear map $\tilde{A}_s$ satisfies $\tilde{A}_s(v)=
s^{\frac{1}{d+1}}v$ for $v\in u^\bot$ and $\tilde{A}_s(u)=s^{\frac{2}{d+1}}u$. Then the image 
under $A_{s^{-1}}$ of 
a cap von $K$ at $y$ converges in the Hausdorff metric as $s\to 0^+$ to a 
cap of the osculating paraboloid of $K$ at $y$. For a more explicit statement, let $A$ be a 
volume preserving affinity of $\R^d$ such that $A(y)=o$ and $A(y-u)=p$, which maps the osculating paraboloid 
of $K$ at $y$ to $\Psi$. Then $\Phi_{s,\beta}:=A\circ A_{\tilde{s}(\beta,s)^{-1}}$ is an affinity satisfying 
$$
\Phi_{s,\beta}(y)=o,\quad \det(\Phi_{s,\beta})=\tilde{s}(\beta,s)^{-1}=
\frac{V(C(\beta))}{V(\widetilde{C}(y,\beta s))},
$$
and, consequently, $\Phi_{s,\beta}(\widetilde{C}(y,\beta s))\to C(\beta)$ in the Hausdorff metric 
as $s\to 0^+$. Moreover, we have
$$
\lim_{s\to 0^+}\Phi_{s,\beta}(\tilde{y}_s)=\lim_{s\to 0^+}\Phi_{s,1}(\tilde{y}_s)=p,
$$
since $\mu(\beta,s)\to 1$ and $\mu(1,s)\to 1$ as $s\to 0^+$, $\tilde{y}_s\in \partial \widetilde{C}(y,s)$ 
and $\Phi_{s,1}(\tilde{y}_s)\in\partial C(1)$, and by \eqref{vertical}. 
Since $\varrho$ is continuous at $y$, the properties of $\Phi_{s,\beta}$ imply that, 
for $i=0,\ldots,k$,
\begin{equation}\label{limit2}
\lim_{s\to 0^+}\PP_{\varrho,\widetilde{C}(y,\beta s)}
\left(\tilde{y}_s\not\in \widetilde{C}(y,\beta s)_{(i)} \right)=
\PP_{{\bf 1},C(\beta)}\left(p\not\in C(\beta)_{(i)} \right).
\end{equation}
We conclude from \eqref{finapprox} and \eqref{limit2}  that
\begin{eqnarray*}
\int_0^{n^{-1/2}}
\PP_\varrho\left(\tilde{y}_s\not\in K_{(n)} \right)s^{-\frac{d-1}{d+1}}\,ds&=&
O\left(\frac{\varepsilon}{n^{\frac2{d+1}}}\right)+\sum_{i=0}^k
\binom{n}{ i}\int_{\frac{\varepsilon^{(d+1)/2}}n}^{\frac{\alpha}n}(\beta s)^i(1-\beta s)^{n-i}\\
&&\qquad \times\,
\PP_{{\bf 1},C(\beta)}\left(p\not\in C(\beta)_{(i)} \right)s^{-\frac{d-1}{d+1}}\,ds.
\end{eqnarray*}
The same formula is obtained for 
$$\int_0^{n^{-1/2}}
\PP_{{\bf 1},B^d}\left(\tilde{w}_s\not\in B^d_{(n)} \right)s^{-\frac{d-1}{d+1}}\,ds,
$$
since $C(\beta)$ is independent of $K$. 
Since $\varepsilon\in(0,1)$ was arbitrary, we conclude
$$
\lim_{n\to\infty}\int_0^{n^{-1/2}}n^{\frac2{d+1}}
\PP_\varrho\left(\tilde{y}_s\not\in K_{(n)} \right)s^{-\frac{d-1}{d+1}}\,ds=
\lim_{n\to\infty}\int_0^{n^{-1/2}}n^{\frac2{d+1}}
\PP_{{\bf 1},B^d}\left(\tilde{w}_s\not\in B^d_{(n)} \right)s^{-\frac{d-1}{d+1}}\,ds.
$$
Now \eqref{ball} yields Lemma~\ref{compareball}.
\proofbox

\noindent{\it Proof of Theorem~\ref{weighted}: } Let  $y\in\partial K$ be a 
normal boundary point of $K$.  
Combining Lemma~\ref{zerocurv}, Lemma~\ref{compareball} 
and (\ref{limitform0}), we obtain 
$$
J_\varrho(y)=c_d\,\varrho(y)^{\frac{-2}{d+1}} \kappa(y)^{\frac1{d+1}}.
$$
Therefore Theorem~\ref{weighted} is implied by (\ref{limitform}).
\proofbox

\section{Polarity and the proof of Theorem~\ref{mainmean}}
\label{secpolar}

In this section, we deduce Theorem \ref{mainmean}  and Theorem \ref{extfacets}  
from Theorem \ref{weighted} and Corollary \ref{weightedcor}, respectively. In 
order to obtain more general results, for not necessarily 
homogeneous or isotropic hyperplane distributions, we 
start with a description of the basic setting.

Let $K\subset\R^d$ be a convex body with $o\in{\rm int}(K)$, 
as usual let $K^*:=\{z\in\R^d:\langle x,z\rangle\le 1\text{ for all }x\in K\}$ 
denote the polar body of $K$, and put $K_1:=K+B^d$. 
Let $\mathcal{H}_K$ denote the set of all hyperplanes $H$ in $\R^d$ for which 
$H\cap \text{int}(K)=\emptyset$ and $H\cap K_1\neq\emptyset$. The 
motion invariant locally finite measure $\mu$ on the space $A(d,d-1)$ of hyperplanes,  
which satisfies $\mu(\mathcal{H}_K)=2$, is explicitly given by 
$$
\mu=2\int_{S^{d-1}}\int_0^\infty\mathbf{1}\{H(u,t)\in\cdot\}\, dt\, \sigma(du),
$$
where $\sigma$ is the rotation invariant probability measure 
on the unit sphere $S^{d-1}$. The model of a random polytope (random polyhedral set) 
described in 
the introduction is based on random hyperplanes with distribution ${\mu}_K
:={2}^{-1}(\mu\llcorner \mathcal{H}_K)$. More generally, we now consider random 
hyperplanes with distribution 
\begin{equation}\label{Defmuq}
{\mu}_q:=\int_{S^{d-1}}\int_0^\infty\mathbf{1}\{H(u,t)\in\cdot\}q(t,u)\, dt\, \sigma(du),
\end{equation}
where  $q:[0,\infty)\times S^{d-1}\to[0,\infty)$ is a measurable function which is 
\begin{enumerate}
\item[(q1)] concentrated on $D_K:=\{(t,u)\in [0,\infty)\times S^{d-1}: h(K,u)\le t\le h(K_1,u)\}$,
\item[(q2)] positive and continuous in a neighborhood of $ \{(t,u)\in [0,\infty)\times S^{d-1}: t=h(K,u)\}$ 
with respect to $D_K$,
\item[(q3)]  and satisfies $\mq(\mathcal{H}_K)=1$.
\end{enumerate}
The intersection of $n$ halfspaces $H_i^-$ 
containing the origin $o$ and bounded by $n$ independent random hyperplanes $H_i$ 
with distribution ${\mu}_q$ is denoted by $K^{(n)}:=\bigcap_{i=1}^nH_i^-$. Probabilities 
and expectations with respect to $\mq$ are denoted by $\PP_{\mq}$ and $\E_{\mq}$, respectively. 
The special example $q\equiv \mathbf{1}_{D_K}$ ($q$ is the characteristic function of $D_K$) covers the situation discussed in the introduction.

In the following, beside the support function, we will also need the radial function  
$\rho(L,\cdot)$ of a convex body $L$ with $o\in{\rm int}(L)$. 
Let $F$ be a nonnegative measurable functional on convex polyhedral sets in $\R^d$. Using 
\eqref{Defmuq} and Fubini's theorem, we get
\begin{align*}
\E_{\mq}(F(K^{(n)}))&=\int_{A(d,d-1)^n}F\left(\bigcap_{i=1}^nH_i^-\right)
\, \mq^{\otimes n}(d(H_1,\ldots,H_n))\\
&=\int_{(S^{d-1})^n}\int_{h(K,u_1)}^{h(K_1,u_1)}\ldots \int_{h(K,u_n)}^{h(K_1,u_n)}
F\left(\bigcap_{i=1}^nH_i^-(u_i,t_i)\right)\prod_{i=1}^nq(t_i,u_i)\,\\
&\qquad\qquad \times dt_n\ldots dt_1\,
\sigma^{\otimes n}(d(u_1,\ldots,u_n)).
\end{align*}
For $t_1,\ldots,t_n>0$, we have
$$
\bigcap_{i=1}^nH_i^-(u_i,t_i)=\left[{t_1}^{-1}u_1,\ldots, {t_n}^{-1}u_n\right]^*.
$$
Using the substitution $s_i=1/t_i$, $\rho(L^*,u_i)=h(L,u_i)^{-1}$ 
for $L\in\mathcal{K}^n$ with $o\in\text{int}(L)$, and polar coordinates, we obtain 
$$
\E_{\mq}(F(K^{(n)}))=\frac{1}{\omega_d^n}\int_{\left(K^*\setminus K_1^*\right)^n}
F([x_1,\ldots,x_n]^*)\prod_{i=1}^n\left(\tilde{q}(x_i)\|x_i\|^{-(d+1)}
\right)\, d(x_1,\ldots,x_n)
$$
with $K^*_1:=(K_1)^*$ and 
$$
\tilde{q}(x):=q\left(\frac{1}{\|x\|},\frac{x}{\|x\|}\right),\qquad x\in K^*\setminus\{o\}.
$$
The case $n=1$ and $F\equiv 1$ yields
$$
\frac{1}{\omega_d}\int_{K^*\setminus K_1^*}
\tilde{q}(x)\|x\|^{-(d+1)}\, dx =1,
$$
hence 
$$
\varrho(x):=\begin{cases}
{\omega_d}^{-1}\tilde{q}(x)\|x\|^{-(d+1)},&
 x\in K^*\setminus K_1^*,\\
0,&x\in K_1^*,
\end{cases}
$$ 
is a probability density with respect to $\mathcal{H}^d\llcorner K^*$ 
which is positive and continuous in a neighborhood of $\partial K^*$ with respect to $K^*$. 
Thus we conclude that
\begin{align*}
\E_{\mq}(F(K^{(n)}))&=\int_{\left(K^*\right)^n}
F([x_1,\ldots,x_n]^*)\prod_{i=1}^n \varrho(x_i)\, d(x_1,\ldots,x_n)\\
&=\E_{\varrho,K^*}\left(F((K_{(n)}^*)^*)\right),
\end{align*}
where $K^*_{(n)}:=(K^*)_{(n)}$. 

\medskip

\begin{prop}\label{equaldistr}
Let $K\subset\R^d$ be a convex body with $o\in{\rm int}(K)$, and let $q$ and $\varrho$ be 
defined as above. Then the random polyhedral sets $K^{(n)}$ and $(K^*_{(n)})^*$ are equal 
in distribution.
\end{prop}

\medskip

For a first application, let
$$
F(P):=\mathbf{1}\{P\subset K_1\}\left(W(P)-W(K)\right),
$$
for a polyhedral set $P\subset\R^d$, with the convention $0\cdot \infty:=0$. 
For $x_1,\ldots,x_n\in K^*\setminus 
K_1^*$, we have $K\subset [x_1,\ldots,x_n]^*$ and, arguing as before,  
\begin{align*}
F([x_1,\ldots,x_n]^*)&=\mathbf{1}\{[x_1,\ldots,x_n]^*\subset K_1\}
\left(W([x_1,\ldots,x_n]^*)-W(K)\right)\\
&=2\cdot \mathbf{1}\{[x_1,\ldots,x_n]^*\subset K_1\}\int_{K^*\setminus[x_1,\ldots,x_n]}\lambda(x)\, dx ,
\end{align*}
where 
$$
\lambda(x):=\begin{cases}
{\omega_d}^{-1}{\|x\|^{-(d+1)}},& x\in K^*\setminus K^*_1,\\
0,&x\in K^*_1.
\end{cases}
$$
Note that if $[x_1,\ldots,x_n]^*\subset K_1$, then 
the set $[x_1,\ldots,x_n]^*$ is bounded, hence $o\in\text{int}([x_1,\ldots,x_n])$, and 
therefore $K_1^*\subset [x_1,\ldots,x_n]^{**}=[x_1,\ldots,x_n]$.

As in \cite{BS08}, it can be shown that $\PP_{\mq}(K^{(n)}\not\subset K_1)\ll \alpha^n$,  
for some $\alpha\in (0,1)$ depending on $K$ and $q$. By Proposition \ref{equaldistr}, we also get
$$
\PP_{\varrho,K^*}\left((K^*_{(n)})^*\not\subset K_1\right)=
\PP_{\mq}\left(K^{(n)}\not\subset K_1\right)\ll\alpha^n.
$$
Hence
\begin{align*}
&\E_{\mq}\left(W(K^{(n)}\cap K_1)-W(K)\right)\\
&\qquad= \E_{\mq}\left(\mathbf{1}\{K^{(n)}\subset K_1\}
\left(W(K^{(n)})-W(K)\right)\right)+O(\alpha^n)\\
&\qquad=2\cdot \E_{\varrho,K^*}\left(\mathbf{1}\{(K^*_{(n)})^*\subset K_1\}
\int_{K^*\setminus K^*_{(n)}}\lambda(x)\, dx\right)
+O(\alpha^n)\\
&\qquad=2\cdot \E_{\varrho,K^*}\left(
\int_{K^*\setminus K^*_{(n)}}\lambda(x)\, dx\right)+O(\alpha^n),
\end{align*}
where we used that $\lambda$ is integrable. 
Therefore, Theorem \ref{weighted} implies
\begin{eqnarray*}
&&\lim_{n\to\infty}n^{\frac2{d+1}}\, \E_{\mq}(W(K^{(n)}\cap K_1)-W(K))\\&&\qquad=
2\cdot \lim_{n\to\infty}
n^{\frac2{d+1}}\, 
\E_{\varrho,K^*}\int_{K^*\setminus K^*_{(n)}}\lambda(x)\,dx\\
&&\qquad=  {2\,{c}_d}
\int_{\partial K^*}\varrho(x)^{-\frac{2}{d+1}}\lambda(x)
\kappa^{*}(x)^{\frac1{d+1}}\,\Hde(dx)\\
&&\qquad=  {2\,{c}_d}\,{\omega_d}^{-\frac{d-1}{d+1}}\, 
\int_{\partial K^*}\tilde{q}(x)^{-\frac{2}{d+1}}\|x\|^{-d+1}\kappa^{*}(x)^{\frac1{d+1}}\,\Hde(dx),
\end{eqnarray*}
where $\kappa^*$ denotes the generalized Gauss curvature of $K^*$. In the following, 
for $x\in \partial K$, 
let $\sigma_K(x)$ denote an exterior unit normal vector of $K$ at $x$. It is unique 
for $\mathcal{H}^{d-1}$ almost all $x\in\partial K$.

\begin{theo}\label{gener1}
Let  $K\subset\R^d$ be a convex body with $o\in{\rm int}(K)$, and let 
$q:[0,\infty)\times S^{d-1}\to[0,\infty)$ be a measurable function 
satisfying (q1)--(q3). Then 
\begin{align}\label{rhs}
&\lim_{n\to\infty}n^{\frac2{d+1}}\, \E_{\mq}(W(K^{(n)}\cap K_1)-W(K))\nonumber\\
&\qquad=
{2\,{c}_d}\,{\omega_d}^{-\frac{d-1}{d+1}}\, 
\int_{\partial K}q(h(K,\sigma_K(x)),\sigma_K(x))^{-\frac{2}{d+1}}
\kappa(x)^{\frac{d}{d+1}}\,\Hde(dx).
\end{align}
\end{theo}

The proof is completed in Section \ref{Polinttrafo} by providing Lemma 
\ref{trafo2}. 

Observe that if $q:\{(h(K,u),u)\in (0,\infty)\times S^{d-1}:u\in S^{d-1}\} 
\to [0,\infty)$ is positive and continuous, then $q$ can be extended to $[0,\infty)\times S^{d-1}$ 
such that (q1)--(q3) are satisfied. For any such extension, the right-hand side of \eqref{rhs} 
remains unchanged. As an example, we may choose $q_1$ such that $q_1(t,u)=t^{(d^2-1)/2}$ for $t=h(K,u)$ and 
$u\in S^{d-1}$. Then the integral in \eqref{rhs} turns into
$$
\int_{\partial K}\frac{\kappa(x)^{\frac{d}{d+1}}}{\langle x,
\sigma_K(x)\rangle^{d-1}}\, \Hde(dx)=\Omega_{d^2}(K),
$$
where
$$
\Omega_p(K):=\int_{\partial K}\frac{\kappa(x)^{\frac{p}{d+p}}}
{\langle x,\sigma_K(x)\rangle^{\frac{(p-1)d}{d+p}}}\, \Hde(dx)
$$
is the $p$-affine surface area of $K$ (see \cite{Lutwak}, \cite{Hug96a}, 
\cite{Hug96b}, \cite{Leichtweiss}, \cite{Werner1}, \cite{Werner2}, \cite{Ludwig0}, 
\cite{Ludwig}). It has been shown that 
$\Omega_{d^2}(K)=\Omega_1(K^*)$; see \cite{Hug96b}. Moreover, for a convex body $L\subset\R^d$, the 
equiaffine isoperimetric inequality states that
$$
\Omega_1(L)\le d\alpha_d^{\frac{2}{d+1}}V(L)^{\frac{d-1}{d+1}}
$$
with equality if and only if $L$ is an ellipsoid (cf.\ \cite{Petty}, \cite{Lutwak0}, \cite{Lutwak}, \cite{Hug96a}, \cite{Boe}). Thus we get 
$$
\lim_{n\to\infty}n^{\frac2{d+1}}\, \E_{\mu_{q_1}}(W(K^{(n)}\cap K_1)-W(K))\le 2dc_d\omega_d^{-\frac{d-1}{d+1}}
\alpha_d^{\frac{2}{d+1}}V(K^*)^{\frac{d-1}{d+1}}
$$
with equality if and only if $K^*$ is an ellipsoid, that is, if and only if $K$ is an ellipsoid.

\bigskip

For another application, we define
$$
F(P):= f_{d-1}(P),
$$
for a convex polyhedral set $P\subset\R^d$. It is well known that 
$f_0(P)=f_{d-1}(P^*)$ for a convex polytope $P\subset\R^d$ with $o\in\text{int}(P)$. 
Thus, from Proposition \ref{equaldistr} we get
\begin{align*}
\E_{\mu_q}\left(f_{d-1}(K^{(n)})\right)
&=\E_{\varrho,K^*}\left(f_{d-1}((K^*_{(n)})^*)\right)\\
&=\E_{\varrho,K^*}\left(\mathbf{1}\{(K^*_{(n)})^*\subset K_1\}
f_{d-1}((K^*_{(n)})^*)\right)\\&\qquad\qquad+\E_{\varrho,K^*}\left(\mathbf{1}
\{(K^*_{(n)})^*\not\subset K_1\}f_{d-1}((K^*_{(n)})^*)\right)\\
&=\E_{\varrho,K^*}\left(\mathbf{1}\{(K^*_{(n)})^*\subset K_1\}
f_{0}(K^*_{(n)})\right)+O(n\cdot \alpha^n)\\
&=\E_{\varrho,K^*}\left(
f_{0}(K^*_{(n)})\right)+O(n\cdot \alpha^n),
\end{align*}
where $\alpha\in(0,1)$ is a suitable constant. 

The following Theorem \ref{gener2} generalizes Theorem \ref{mainmean} in the same way as
Theorem \ref{gener1} extends Theorem \ref{extfacets}.

\begin{theo}\label{gener2}
Let  $K\subset\R^d$ be a convex body with $o\in{\rm int}(K)$, and let 
$q:[0,\infty)\times S^{d-1}\to[0,\infty)$ be a measurable function 
satisfying (q1)--(q3). Then 
$$
\lim_{n\to\infty}n^{-\frac{d-1}{d+1}}\, \E_{\mq}(f_{d-1}(K^{(n)}))
={{c}_d}\,{\omega_d}^{-\frac{d-1}{d+1}}\, 
\int_{\partial K}q(h(K,\sigma_K(x)),\sigma_K(x))^{\frac{d-1}{d+1}}
\kappa(x)^{\frac{d}{d+1}}\,\Hde(dx).
$$
\end{theo}

The proof follows by applying Corollary \ref{weightedcor} and  Lemma 
\ref{trafo2}. 

\section{Polarity and an integral transformation}\label{Polinttrafo}

In this section, we establish the required integral transformation involving the Gauss curvatures 
of a convex body and its polar body.

Let $L\subset \R^d$ be a convex body.  If the support function $h_L$ of $L$ 
is differentiable at $u\neq o$,  then the gradient $\nabla h_L(u)$ 
of $h_L$  at $u$ is equal to the unique boundary point of $L$ having $u$ as an exterior normal 
vector. In particular, the gradient of $h_L$ is a function which is homogeneous of degree zero. 
Note that $h_L$ is differentiable at $\Hde$ almost all unit vectors. 
We write $D_{d-1}h_L(u)$ for the product of the principal radii of curvature 
of $L$ in direction $u\in S^{d-1}$, whenever the support function $h_L$ is 
twice differentiable in the generalized sense at $u\in S^{d-1}$. Note that this 
is the case for $\mathcal{H}^{d-1}$ almost all $u\in S^{d-1}$. The Gauss map 
$\sigma_L$ is defined $\mathcal{H}^{d-1}$ almost everywhere on $\partial L$. If $\sigma_L$ is 
differentiable in the generalized sense at $x\in\partial L$, which is the case 
for $\mathcal{H}^{d-1}$ almost all $x\in\partial L$, then the product of the 
eigenvalues of the differential is the Gauss curvature $\kappa_L(x)$. The 
connection to curvatures defined on the generalized normal bundle $\NL$ of $L$ 
will be used in the following proof (cf.\ \cite{Hug98}).

\begin{lemma}\label{trafo1}
Let $L\subset\R^d$ be a convex body containing the origin in its interior. If  
$g:\partial L\to[0,\infty]$ is measurable, then
$$
\int_{\partial L}g(x)\kappa_L(x)^{\frac{1}{d+1}}\, \Hde(dx)
=\int_{S^{d-1}}g(\nabla h_L(u))
D_{d-1}h_L(u)^{\frac{d}{d+1}}\,\Hde(du).
$$
\end{lemma}

\begin{proof} 
In the following proof, we use results and methods from \cite{Hug98}, to which 
we refer for additional references and detailed definitions. Let $\NL$ denote the 
generalized normal bundle of $L$, and let $k_i(x,u)\in[0,\infty]$, $i=1,\ldots,d-1$, be 
the generalized curvatures of $L$, which are defined for $\mathcal{H}^{d-1}$ almost 
all $(x,u)\in \NL$. Expressions such as
$$
\frac{k_i(x,u)^{\frac{1}{d+1}}}{\sqrt{1+k_i(x,u)^2}}\qquad\text{or}\qquad 
\frac{k_i(x,u)}{\sqrt{1+k_i(x,u)^2}}
$$
with $k_i(x,u)=\infty$ are understood as limits as $k_i(x,u)\to\infty$, and yield $0$ or 
$1$, respectively in the two given examples. As is common in measure theory, the product 
$0\cdot\infty$ is defined as $0$. 

Our starting point is the expression
\begin{equation}\label{starter}
I:=\int_{\NL}g(x)\prod_{i=1}^{d-1}
\frac{k_i(x,u)^{\frac{1}{d+1}}}{\sqrt{1+k_i(x,u)^2}}\, \mathcal{H}^{d-1}(d(x,u)),
\end{equation}
which will be evaluated in two different ways. A comparison of the resulting 
expressions yields the assertion of the lemma. 

First, we rewrite $I$ in the form
\begin{equation}\label{I1}
I=\int_{\NL}g(x)
\left(\prod_{i=1}^{d-1}k_i(x,u)\right)^{-\frac{d}{d+1}}J_{d-1}\pi_2(x,u)\, 
\mathcal{H}^{d-1}(d(x,u)),
\end{equation}
where
$$
J_{d-1}\pi_2(x,u)=\prod_{i=1}^{d-1}\frac{k_i(x,u)}{\sqrt{1+k_i(x,u)^2}},
$$
for $\mathcal{H}^{d-1}$ almost all $(x,u)\in\NL$, 
is the (approximate) Jacobian of the map $\pi_2:\NL\to S^{d-1}$, $(x,u)\mapsto u$. To check \eqref{I1}, 
we distinguish the following cases. If $k_i(x,u)=0$ for some $i$, then the integrands on the 
right-hand sides of \eqref{starter} and   of \eqref{I1} are zero, since 
$0\cdot\infty=0$ and  $J_{d-1}\pi_2(x,u)=0$. If $k_i(x,u)\neq 0$ for all $i$ and $k_j(x,u)=\infty$ for some 
$j$, then again both integrands are zero. In all other cases the assertion is clear. 

For $\mathcal{H}^{d-1}$ almost all $u\in S^{d-1}$,  $\nabla h_L(u)\in \partial L$ is the unique boundary 
point of $L$ which has $u$ as an exterior unit normal vector. Then the coarea formula yields
$$
I=\int_{S^{d-1}}g(\nabla h_L(u))\left(\prod_{i=1}^{d-1}k_i(\nabla h_L(u),u)\right)^{-\frac{d}{d+1}}\, 
\mathcal{H}^{d-1}(du).
$$
Using Lemma 3.4 in \cite{Hug98}, we get
\begin{equation}\label{rep1}
I=\int_{S^{d-1}}g(\nabla h_L(u))D_{d-1}h_L(u)^{\frac{d}{d+1}}\,
\mathcal{H}^{d-1}(du).
\end{equation}
Now we consider also the projection $\pi_1:\NL\to\partial L$, $(x,u)\mapsto x$, which has the 
(approximate) Jacobian 
$$
J_{d-1}\pi_1(x,u)=\prod_{i=1}^{d-1}\frac{1}{\sqrt{1+k_i(x,u)^2}},
$$
for $\mathcal{H}^{d-1}$ almost all $(x,u)\in\NL$. A similar argument as before yields
\begin{align*}
I&=\int_{\NL}g(x)
\left(\prod_{i=1}^{d-1}k_i(x,u)\right)^{\frac{1}{d+1}}J_{d-1}\pi_1(x,u)\, 
\mathcal{H}^{d-1}(d(x,u))\\
&=\int_{\partial L}g(x)\left(\prod_{i=1}^{d-1}k_i(x,\sigma_L(x))
\right)^{\frac{1}{d+1}}\, \mathcal{H}^{d-1}(dx).
\end{align*}
By Lemma 3.1 in \cite{Hug98}, we finally also get
\begin{equation}\label{rep2}
I=\int_{\partial L}g(x)\kappa_L(x)^{\frac{1}{d+1}}\, \mathcal{H}^{d-1}(dx).
\end{equation}
A comparison of equations \eqref{rep1} and \eqref{rep2} gives the required equality.
$\Box$
\end{proof}

\bigskip

\noindent
{\bf Remark}  An alternative argument can be based on 
arguments similar to those used in \cite{Hug96a} for the proof of 
the equality of two representations of the affine surface area 
of a convex body.

\bigskip

\begin{lemma}\label{trafo2}
Let $K\subset\R^d$ be a convex body with $o\in\text{\rm int}(K)$. If  
$f:[0,\infty)\times S^{d-1}\to[0,\infty)$ is a measurable function and 
$\tilde{f}(x):=f\left(\|x\|^{-1},\|x\|^{-1}x\right)$, $x\in \partial K^*$, then 
$$
\int_{\partial K^*}\tilde{f}(x)\|x\|^{-d+1}\kappa^*(x)^{\frac{1}{d+1}}\, \Hde(dx)
=\int_{\partial K}f(h(K,\sigma_K(x)),\sigma_K(x))\kappa(x)^{\frac{d}{d+1}}\, \Hde(dx).
$$
\end{lemma}

\begin{proof} We apply Lemma \ref{trafo1} with $L=K^*$ and 
$g(x)=\tilde{f}(x)\|x\|^{-d+1}$, $x\in \partial K^*$, 
and thus we get
\begin{align*}
&\int_{\partial K^*}\tilde{f}(x)\|x\|^{-d+1}\kappa^*(x)^{\frac{1}{d+1}}\, \Hde(dx)\\
&\qquad =
\int_{S^{d-1}}\tilde{f}(\nabla h_{K^*}(u))\|\nabla h_{K^*}(u)\|^{-d+1}
D_{d-1}h_{K^*}(u)^{\frac{d}{d+1}}\, \Hde(du).
\end{align*}
Next we apply Theorem 2.2 in \cite{Hug96b} (or the second part of 
Corollary 5.1 in \cite{Hug2002}). Thus, using the fact that, for $\Hde$ almost all $u\in S^{d-1}$, 
$h_{K^*}$ is differentiable in the generalized sense at $u$ and $\rho(K,u)u$ is a normal 
boundary point of $K$, we have
$$
D_{d-1}h_{K^*}(u)^{\frac{d}{d+1}}=\kappa(x)^{\frac{d}{d+1}}\langle u,\sigma_K(x)\rangle^{-d},
$$
where $x=\rho(K,u)u\in\partial K$ and $u=\|x\|^{-1}x\in S^{d-1}$. Thus we obtain
\begin{align*}
&\int_{\partial K^*}\tilde{f}(x)\|x\|^{-d+1}\kappa^*(x)^{\frac{1}{d+1}}\, \Hde(dx)\\
&\qquad =
\int_{S^{d-1}}\tilde{f}(\nabla h_{K^*}(u))\frac{\|\nabla h_{K^*}(u)\|^{-d+1}}{\langle u,\sigma_K(\rho(K,u)u)\rangle^d}\,\kappa(\rho(K,u)u)^{\frac{d}{d+1}}\, \Hde(dx).
\end{align*}
The bijective and bilipschitz transformation $T:S^{d-1}\to \partial K$, $u\mapsto \rho(K,u)u$, 
has the Jacobian
$$
JT(u)=\frac{\|\nabla h_{K^*}(u)\|}{h_{K^*}(u)^d}
$$
for $\Hde$ almost all $u\in S^{d-1}$ (see the proof of Lemma 2.4 in \cite{Hug96b}). Thus
\begin{align*}
&\int_{\partial K^*}\tilde{f}(x)\|x\|^{-d+1}\kappa^*(x)^{\frac{1}{d+1}}\, \Hde(dx)\\
&\qquad =
\int_{\partial K}\tilde{f}\left(\nabla h_{K^*}\left(\frac{x}{\|x\|}\right)\right)
\frac{\|\nabla h_{K^*}\left(\frac{x}{\|x\|}\right)\|^{-d}}{\langle \frac{x}{\|x\|},\sigma_K(x)\rangle^d}h_{K^*}\left(\frac{x}{\|x\|}\right)^d
\kappa(x)^{\frac{d}{d+1}}\, \Hde(dx)\\
&\qquad =\int_{\partial K}\tilde{f}\left(\nabla h_{K^*}\left({x}\right)\right)
\frac{\|\nabla h_{K^*}\left(x\right)\|^{-d}}{\langle x,\sigma_K(x)
\rangle^d}h_{K^*}\left(x\right)^d
\kappa(x)^{\frac{d}{d+1}}\, \Hde(dx)\\
&\qquad =\int_{\partial K}{f}(\|\nabla h_{K^*}(x)\|^{-1}, \nabla h_{K^*}(x)/\|\nabla h_{K^*}(x)\| )
\kappa(x)^{\frac{d}{d+1}}\, \Hde(dx),\\
&\qquad =\int_{\partial K}{f}(h_K(\sigma_K(x)),\sigma_K(x))
\kappa(x)^{\frac{d}{d+1}}\,\Hde(dx),\\
\end{align*}
since $h_{K^*}(x)=1$ for $x\in\partial K$ and $x^*:=\nabla h_{K^*}(x)$ satisfies 
$\|x^*\|^{-1}=\langle x,\sigma_K(x)\rangle$ and $x^*/\|x^*\|=\sigma_K(x)$, for $\Hde$ 
almost all $x\in \partial K$. $\Box$
\end{proof}

\vspace*{1cm}

\noindent{\bf Acknowledgement } We are grateful for 
stimulating discussions with Rolf Schneider.

\vspace{3mm}

\small

\noindent K\'aroly J. B\"or\"oczky,\newline 
Alfr\'ed R\'enyi Institute of Mathematics
Hungarian Academy of Sciences,\newline 
 PO Box 127,
H--1364 Budapest, Hungary\newline
{\it E-mail:}
carlos@renyi.hu\newline
and\newline
Department of Geometry, Roland E\"otv\"os University\newline
P\'azm\'any P\'eter s\'et\'any 1/C,
H-1117 Budapest, Hungary\newline

\noindent Ferenc Fodor\newline
Department of Geometry, University of Szeged,\newline
Aradi v\'ertan\'uk tere 1, H-6720 Szeged, Hungary\newline
{\it E-mail:} fodorf@math.u-szeged.hu\newline
and\newline
Department of Mathematics and Statistics\newline
University of Calgary\newline
2500 University Dr. N.W.\newline
Calgary, Alberta, Canada\newline
T2N 1N4\newline
{\it Email:} fodor@math.ucalgary.ca\newline

\noindent Daniel Hug,\newline 
Institut f\"ur Algebra und Geometrie, 
Universit\"at Karlsruhe (TH), KIT, \newline 
D-76133 Karlsruhe, 
Germany \newline
{\it E-mail:} daniel.hug@kit.edu

\end{document}